\documentclass[12pt,a4paper,notitlepage]{report}
\usepackage[utf8]{inputenc}
\usepackage[english]{babel}

\usepackage{comment}
\usepackage{enumitem}
\usepackage{tabularx}

\usepackage{makeidx}
\makeindex

\usepackage{amsmath}
\usepackage{amsfonts}
\usepackage{amssymb}
\usepackage{amsthm}
\usepackage{tikz-cd}
\usepackage{quiver}
\usepackage{mathtools}
\usepackage{ stmaryrd }
\usepackage{calligra} 

\usepackage{varioref}
\usepackage{hyperref}
\usepackage[capitalise]{cleveref}

\usepackage[style=numeric]{biblatex}
\usepackage{csquotes}

\theoremstyle{plain}
\newtheorem{theorem}{Theorem}[subsection] 

\theoremstyle{definition}
\newtheorem{definition}[theorem]{Definition} 
\newtheorem{proposition}[theorem]{Proposition}
\newtheorem{propdef}[theorem]{Proposition/Definition}
\newtheorem{defprop}[theorem]{Definition/Proposition}

\newtheorem{corollary}[theorem]{Corollary}
\newtheorem{lemma}[theorem]{Lemma}
\newtheorem{remark}[theorem]{Remark}
\newtheorem{warning}[theorem]{Warning}
\newtheorem{example}[theorem]{Example} 

\newtheorem{question}[theorem]{Question}
\newtheorem{answer}[theorem]{Answer}
\newtheorem{observation}[theorem]{Observation}
\newtheorem{conjecture}[theorem]{Conjecture}
 
\newcommand{\fun}{\text{Fun}}
\newcommand{\poly}{\text{Poly}}
\newcommand{\Poly}{\text{Poly}}

\newcommand{\yo}{\text{y}}

\newcommand{\cj}{\mathcal{J}}
\newcommand{\cjtop}{{\cj_{top}}}
\newcommand{\cjenr}{\cj_{enr}}

\newcommand{\cs}{\mathcal{S}}
\newcommand{\ct}{\mathcal{T}}
\newcommand{\csp}{\cs_*}

\newcommand{\spc}{Sp}

\newcommand{\ce}{\fun(\cj,\cs)}

\newcommand{\cC}{\mathcal{C}}
\newcommand{\cD}{\mathcal{D}}

\newcommand{\lan}{\text{Lan}}
\newcommand{\mor}{\text{mor}}

\newcommand{\conv}{\circledast}
\newcommand{\redconv}{\circledast_\text{red}}

\newcommand{\im}{\text{im}}

\newcommand{\colim}{\operatorname{colim}\displaylimits}
\newcommand{\B}{\text{B}}

\newcommand{\BO}{\text{BO}}
\newcommand{\BU}{\text{BU}}
\newcommand{\TOP}{\text{TOP}}
\newcommand{\BTOP}{\text{BTOP}}
\newcommand{\into}{\hookrightarrow}
\newcommand{\Th}{\text{Th}}
\newcommand{\Cat}{\text{Cat}}
\newcommand{\nfin}{\text{N}(\text{Fin}_*)}

\newcommand{\Map}{\text{Map}}
\newcommand{\fib}{\text{fib}}

\newcommand{\calg}{\text{CAlg}}
\newcommand{\homog}{\text{Homog}^n(\cj,\csp)}
\newcommand{\Homog}{\text{Homog}^n}

\newcommand{\sh}{\text{sh}}
\newcommand{\id}{\text{id}}
\newcommand{\Nat}{\text{Nat}}

\newcommand{\inthom}{\mathcal{F}}
\newcommand{\Lan}{\text{Lan}}
\newcommand{\hofib}{\mathrm{hofib}}
\newcommand{\Tow}{\mathrm{Tow}}

\bibliography{references.bib}

\author{Leon Hendrian}
\title{Monoidal Structures in Orthogonal Calculus}
\date{}

\begin{document}

\maketitle
\begin{abstract}
Orthogonal Calculus, first developed by Weiss in 1991 \cite{weiss-oc}, provides a \enquote{calculus of functors} for functors from real inner product spaces to spaces. Many of the functors to which Orthogonal Calculus has been applied since carry an additional lax symmetric monoidal structure which has so far been ignored. For instance, the functor $V \mapsto \BO(V)$ admits maps
$$\BO(V) \times \BO(W) \to \BO(V \oplus W)$$
which determine a lax symmetric monoidal structure. 

Our first main result, \cref{approx_of_monoidal_functors}, states that the Taylor approximations of a lax symmetric monoidal functor are themselves lax symmetric monoidal. 
We also study the derivative spectra of lax symmetric monoidal functors, and prove in \cref{final_derivative_spectra_result} that they admit $O(n)$-equivariant structure maps of the form
$$\Theta^nF \otimes \Theta^nF \to D_{O(n)} \otimes \Theta^nF$$
where $D_{O(n)} \simeq S^{\text{Ad}_n}$ is the Klein-Spivak dualising spectrum of the topological group $O(n)$. 

As our proof methods are largely abstract and $\infty$-categorical, we also formulate Orthogonal Calculus in that language before proving our results. 

This article is largely identical to the authors PhD thesis \cite{hendrian-thesis}.

\end{abstract}

\tableofcontents
\newpage
\thispagestyle{empty}
\section*{Acknowledgements}
I am grateful to my advisor Michael Weiss for all the support I received during my time as a PhD student. Further thanks go out to Thomas Nikolaus for showing me the way to $\infty$-land. \\
The author was supported by the Deutsche Forschungsgemeinschaft (DFG, German Research Foundation) – Project-ID 427320536 – SFB 1442, as well as under Germany’s Excellence Strategy EXC 2044 – 390685587, Mathematics Münster: Dynamics–Geometry–Structure.

\chapter{Introduction} 
\label{chapter_introduction}

\section{Overview}
\sectionmark{Overview}

As already mentioned in the abstract, the main results are
\begin{itemize}
\item \cref{orthogonal_monoidality_theorem} and \cref{approx_of_monoidal_functors}, which state that the polynomial approximation functors $T_n$ are lax symmetric monoidal and hence that $T_nF$ of a lax symmetric monoidal functor $F$ is lax symmetric monoidal.
\item \cref{final_derivative_spectra_result}, which states that the derivative spectra $\Theta^nF$ of a lax symmetric monoidal functor $F \colon \cj \to (\csp,\times)$ admit structure maps
$$\Theta^nF \otimes \Theta^n F \to \Theta^n F \otimes D_{O(n)}$$
\end{itemize} 

The content is structured into chapters as follows:

\begin{itemize}
\item In \cref{chapter_introduction}, we first briefly introduce Orthogonal Calculus (and functor calculus in general) in \cref{section_what_is_oc} and then explain the central motivation for this work in \cref{section_monoidal_structures}, particularly in \cref{day_convolution_exposition}.
\item In \cref{chapter_survey}, we mention some interesting work on or involving Orthogonal Calculus in the literature. This chapter is not essential to the thesis, but as no existing survey on Orthogonal Calculus is known to the author, it might be useful to a reader who wants to know more about Orthogonal Calculus in general.
\item In \cref{infinity_oc}, we go over essential features of Orthogonal Calculus in $\infty$-categorical language. For the proofs of the basic properties, we mostly refer back to \cite{weiss-oc} and then explain how the $\infty$-categorical formulation follows.
\item In \cref{chapter_monoidal_oc_theorem}, we examine the Day Convolution symmetric monoidal structure on $\fun(\cj,\cs)$ in detail. A particularly important feature is the existence of an internal hom functor 
$$\inthom(-,-) \colon \fun(\cj,\cs) \times \fun(\cj,\cs) \to \fun(\cj,\cs)$$
whose properties will play a major role in \cref{chapter_monoidal_oc_theorem} and \cref{monoidal_oc_spectra}. At the end of this chapter we prove \cref{orthogonal_monoidality_theorem}.
\item In \cref{monoidal_oc_spectra}, we show that the homogeneous layers
$$L_n F = \text{fib}(T_nF \to T_{n-1}F) \colon \cj \to \csp$$
of a lax symmetric monoidal functor
$$F \colon \cj \to \csp$$
are also lax symmetric monoidal. Note that $\csp$ is equipped with the cartesian product monoidal structure here (although the same would be true for the smash product monoidal structure). We then use the equivalence
\begin{align*}
\spc^{\BO(n)} &\to \homog \\
\Theta &\mapsto \left( V \mapsto \Omega^\infty \left(S^{nV} \wedge \Theta \right)_{hO(n)} \right)
\end{align*}
to calculate the derivatives of the internal hom functor and the convolution product of homogeneous functors. This enables us to prove \cref{final_derivative_spectra_result}, i.e.\ the existence of the maps
$$\Theta^n F \otimes \Theta^n F \to \Theta^n F \otimes D_{O(n)}$$
\item In the final chapter, \cref{chapter_applications_and_open_questions}, we mention a few possible examples as well as a (conjectured) alternative description of a multiplication on $\Theta^1$. We finish by stating some open questions.
\end{itemize}

\section{What is Orthogonal Calculus?}
\label{section_what_is_oc}

\subsection{Functor Calculus in General}
In homotopy theory, the objects of interest often are, or can be viewed as, functors $F \colon \cC \to \cD$. There are multiple theories which are (or could be) referred to as \textit{Functor Calculus}. Among these are 
\begin{itemize}
\item \textbf{Goodwillie Calculus} studies functors from spaces to spaces\footnote{
More generally, one can also apply it to functors $F \colon \cC \to \cD$ where $\cC$ is an $\infty$-category with finite colimits and $\cD$ admits finite limits as well as sequential colimits such that the formation of sequential colimits commutes with finite limits. See \cite[Chapter~6]{lurie-ha}.
}. (The original sources are \cite{GoodwillieI, GoodwillieII, GoodwillieIII}. See also \cite[Chapter~6]{lurie-ha} for a modern generalisation using $\infty$-categories.)
\item \textbf{Embedding Calculus} and \textbf{Manifold Calculus} (\cite{goodwillie-weiss-embeddings, weiss-embeddings}) study functors from a category of manifolds (possibly restricted to submanifolds of a given manifold) and embeddings to spaces.
\item \textbf{Orthogonal Calculus}, the subject of this work, studies functors from the category of real vector spaces with inner products and linear isometric maps to spaces. 
 (\cite{weiss-oc} is the original source, a model category version has been developed by Barnes and Oman \cite{barnes-oman}.)
\end{itemize}
The guiding principle of any of these calculi is the following: 
\begin{itemize}

\item There are categories\footnote{
sometimes very concrete, sometimes less so, usually enriched categories or $\infty$-categories}
 $\cC$ and $\cD$. One is interested in functors (usually specific ones) $F \colon \cC \to \cD$.
 
\item There is a notion of $n$-polynomiality of these functors, in particular of the subcategories of functors which are polynomial of degree $\leq n$
$$\text{Poly}^{\leq n}(\cC,\cD) \subset \fun(\cC,\cD)$$
Any functor $F$ admits universal $n$-polynomial approximations denoted\footnote{
The notation used in the literature is unfortunately not consistent across calculi. In Goodwillie Calculus, the approximations are usually called $P_nF$, in Orthogonal Calculus $T_nF$. To make matters worse, $T_nF$ is also used in the Goodwillie Calculus literature for something that is called $\tau_nF$ in Orthogonal Calculus. We stick to the conventions from Orthogonal Calculus.} 
$F \to T_nF$ for any $n \in \mathbb{N}$. These $T_n$ determine left-adjoints to the inclusions, i.e.\

\[\begin{tikzcd}
	{\poly^{\leq n}(\cC,\cD)} && {\fun(\cC,\cD)}
	\arrow[""{name=0, anchor=center, inner sep=0}, "{T_n}"', curve={height=12pt}, from=1-3, to=1-1]
	\arrow[""{name=1, anchor=center, inner sep=0}, "\iota"', curve={height=12pt}, from=1-1, to=1-3]
	\arrow["\dashv"{anchor=center, rotate=-90}, draw=none, from=0, to=1]
\end{tikzcd}\]

$T_nF$ is also called the {\it $n$-th Taylor approximation of $F$}. 

\item There are inclusions 
$$\poly^{\leq n-1}(\cC,\cD) \subset \poly^{\leq n}(\cC,\cD)$$
Hence, the $n$-polynomial approximations assemble into a diagram of the form
\[\begin{tikzcd}
	&&& {T_\infty F := \lim_n T_nF} \\
	&&& {\vdots} \\
	&&& {T_1F} \\
	{F} &&& {T_0F}
	\arrow["{\eta_0}", from=4-1, to=4-4]
	\arrow["{\eta_1}", from=4-1, to=3-4, curve={height=-6pt}]
	\arrow[from=4-1, to=2-4, curve={height=-6pt}, dotted]
	\arrow["{\lim_n \eta_n}", from=4-1, to=1-4, curve={height=-12pt}]
	\arrow[from=1-4, to=2-4]
	\arrow[from=2-4, to=3-4]
	\arrow[from=3-4, to=4-4]
\end{tikzcd}\]

This tower is called the {\it Taylor tower}\footnote{
Depending on the type of functor calculus, other names are also used, e.g.\ \textit{Goodwillie tower} in Goodwillie Calculus, \textit{Orthogonal tower} or \textit{Weiss tower} in Orthogonal Calculus.
} of $F$.

\item A functor $F$ is said to be $n$-homogeneous if it is $n$-polynomial and $T_{n-1}F$ vanishes\footnote{
This means that it is equivalent to the constant functor to the terminal object.
}. The most important general examples are the fibres of the maps $T_nF \to T_{n-1}F$,  which are always $n$-homogeneous. These are called the homogeneous layers of $F$.

\item An important but often difficult question is if a functor $F$ is \textit{analytic} in the sense that
$$ F \cong \lim_n T_nF $$
i.e.\ if it can be reconstructed from its Taylor approximations.
Known criteria usually involve connectivity estimates for the maps in the image of $F$.

\item In Goodwillie and Orthogonal Calculus, there is a classification theorem for $n$-homogeneous functors in terms of spectra with actions of the symmetric group $\Sigma_n$ resp.\ the orthogonal group $O(n)$. The spectrum classifying the fibre of $T_nF \to T_{n-1}F$ is called the $n$-th \textit{derivative} of $F$.
\end{itemize}

\begin{remark}[Analogy to calculus of real functions] $ $ \\
Functor calculi are partially motivated by a surprisingly strong analogy with calculus for smooth functions $f \colon \mathbb{R} \to \mathbb{R}$. Under this analogy, $n$-polynomial functors correspond to polynomial functions of degree at most $n$. The $n$-polynomial approximation $F \to T_nF$ corresponds to replacing a function $f$ by its $n$-th Taylor polynomial. The homogeneous layers of $F$ correspond to the difference of the $n$-th and the $n-1$-th Taylor polynomial (at $0$), i.e.\ the expression 
$$\frac{f^{(n)}(0)}{n!} \cdot x^n$$
which is obviously closely related to the derivative $f^{(n)}(0)$. 

Even the division by $n!$ corresponds to something in functor calculus: In Goodwillie calculus, $n$-homogeneous functors have the form
$$X \mapsto \Omega^\infty \left( \left(\partial_n(F) \otimes X^{\wedge n}\right)_{h\Sigma_n} \right)$$
where $\partial_n(F)$ is a spectrum with an action of the symmetric group $\Sigma_n$. The homotopy orbit construction $(-)_{h\Sigma_n}$ can be thought of as dividing out the action. As $|\Sigma_n|=n!$, this can be seen as analogous to dividing by $n!$. 

A similar thing is true in Orthogonal Calculus. $n$-homogeneous functors take the form
$$V \mapsto \Omega^\infty \left( \left(
S^{V \otimes \mathbb{R}^n} \wedge \Theta^nF \right)
_{hO(n)} \right)$$
where $\Theta^nF$ is a spectrum with an action of the orthogonal group $O(n)$. 
These spectra $\partial_n(F) \in \spc^{B\Sigma_n}$ and $\Theta^nF \in \spc^{\BO(n)}$ are referred to as the derivative spectra of $F$.
\end{remark}

\subsection{Orthogonal Calculus}
Orthogonal Calculus was first introduced by Michael Weiss in 1995 (\cite{weiss-oc}). It is a type of functor calculus in the previously discussed sense:

\begin{itemize}
\item $\cj$ denotes the (topological\footnote{
or, in later versions, simplicially enriched category, or quasicategory, via the appropriate nerve constructions
}) category of finite-dimensional real inner product vector spaces where the morphism spaces are spaces of linear isometric (in particular injective) maps (also known as Stiefel manifolds).

\item The functors that one studies are always defined on $\cj$, typically with codomain the category of spaces\footnote{
Parts of the literature are ambiguous if this means topological spaces, a \textit{good} category of spaces, or something else, and in which sense the functors are \textit{continuous}. In the later parts of this work, we will consider functors of $\infty$-categories to the $\infty$-category $\cs$ of spaces.
} or the category of pointed spaces. Other possibilities for the target category are the category of spectra (implicit in \cite{weiss-oc}) or localisations of it \cite{barnes-rational, taggart-local}.

\item A functor
$$F \colon \cj \to \cs$$
is called $n$-polynomial if the natural maps
$$ F(V) \to \lim_{0 \neq U \subseteq \mathbb{R}^{n+1}} F(V \oplus U)$$
which are induced by the inclusions
$V \hookrightarrow V \oplus U$ for $U \subseteq \mathbb{R}^{n+1}, U \neq 0$
are homotopy equivalences. \\
The limit on the right-hand side is not indexed by a (discrete) poset, but is a (homotopy) limit over the space / simplicial set of subspaces $0 \neq U \subseteq \mathbb{R}^{n+1}$. In the $\infty$-categorical world, the meaning of this is clear. 
In the more classical world, most homotopy limits are formed over a indexing diagram which is a $1$-category and not a space, so some care has to be taken (see e.g.\ \cite[Comment to Def.~5.1]{weiss-oc}).
\item An arbitrary functor $F$ can be made $n$-polynomial by replacing $F$ iteratively by $\tau_nF := \lim_{0 \neq U \subseteq \mathbb{R}^{n+1}} F(V \oplus U)$, i.e.\ by
$$T_nF = \colim \left( F \to \tau_nF \to \tau_n(\tau_nF) \to \ldots \right)$$
which spells out to the colimit over the sequential diagram
$$\lim_{0 \neq U_1 \subseteq \mathbb{R}^{n+1}} F(V \oplus U) \to \lim_{0 \neq U_2 \subseteq \mathbb{R}^{n+1}} \lim_{0 \neq U_1 \subseteq \mathbb{R}^{n+1}} F(V \oplus U_1 \oplus U_2) \to \ldots$$
\item An important role is played by the $n$-\textit{homogeneous functors}, which are functors $F$ which are $n$-polynomial and such that $T_{n-1}$ vanishes. They occur in particular as the layers\footnote{i.e.\ fibres of the maps $T_nF \to T_{n-1}F$}
in the tower, and correspond to spectra with $O(n)$-actions in the following way
\begin{theorem}[Weiss \cite{weiss-oc}, Theorem 7.3.]
Any $n$-homogeneous functor to $\csp$, in particular the layers of the tower, is of the form
$$\Omega^\infty \left( (S^{\mathbb{R}^n \otimes V} \wedge \Theta^n F)_{h O(n)} \right)$$
for some spectrum with $O(n)$-action $\Theta^nF \in Sp^{BO(n)}$.
\end{theorem}
\item The \textit{derivative spectra} $\Theta^nF$ can also be constructed more explicitly out of a functor $F$. A very brief summary of this is as follows:
\begin{itemize}
\item There are categories $\cj_0 \subset \ldots \subset \cj_n$ enriched in pointed spaces, and $F^{(n)} \colon \cj_{n} \to \cs$, the \enquote{unstable $n$-th derivative} of $F$, is defined as a Kan extension of $F \colon \cj_0 \to \cs$ along $\cj_0  \to \cj_{n}$, and $\Map_{\cj_0}(V,W) = \Map_\cj(V,W)_+$. (This is in \cite{weiss-oc}, mainly section 2.)
\item $F^{(n)}$ relates to the derivative spectra $\Theta^nF$ by $(\Theta^nF)_{nl} = F^{(n)}(\mathbb{R}^l)$. (See e.g.\ \cite[p.2, p.3, p.10, p.12]{weiss-oc}.)
\item There are fibre sequences
$$F^{(n)}(V) \to F(V) \to \lim_{0 \neq U \subseteq \mathbb{R}^{n}} F(U \oplus V)$$
(\cite{weiss-oc}, Proposition 5.3)
\end{itemize}

\end{itemize}

A much more detailled description of Orthogonal Calculus follows in \cref{infinity_oc}.

\section{Monoidal Structures}
\label{section_monoidal_structures}

Our central question is
\begin{center}
{\it Which \enquote{product rules} hold in Orthogonal Calculus?}
\end{center}

To make sense of the question, one first has to specify the products involved. They take the form of symmetric monoidal structures.

\begin{observation}
There are the following relevant symmetric monoidal structures:
\begin{itemize}
\item On $\cj$, there is a \enquote{direct sum} symmetric monoidal structure \enquote{$\oplus$}, given by direct sum of underlying vector spaces and block sum of scalar products.\footnote{This is not actually a coproduct in $\cj$, as morphism spaces in $\cj$ do not consist of all linear maps but only of isometric ones. See \cref{direct-sum-on-cj-is-not-coproduct}.}

\item On $\cs$, there is the cartesian monoidal structure given by products of spaces.
\item On $\cs_*$, there is the cartesian product, as well as the smash product, of pointed spaces.
\item On $\spc$, there is the tensor product (also known as smash product) of spectra.
\end{itemize}
\end{observation}

Let $(\ct,\otimes)$ be one of the monoidal categories $\{(\cs,\times), (\cs,\times), (\csp,\wedge),(\spc,\otimes)\}$. 
There are (at least) two natural induced monoidal structures on $\fun(\cj,\ct)$. We refer to them as the \enquote{pointwise monoidal structure} and the \enquote{Day convolution monoidal structure} and shall explain them shortly.

The case of pointwise products has been treated (with some restrictions) by Hahn and Yuan \cite[Section~4]{hahn-yuan}. We will quickly explain their result, before discussing the case of convolution products, which is the focus of our work.

\subsection{The Pointwise Monoidal Structure}
Using only the monoidal structure on $\ct$ (and ignoring that on $\cj$), there is a monoidal structure given by
$$(F \otimes_{\fun(\cj,\ct)}G)(V) = F(V) \otimes_\ct G(V)$$
(For formal details, see for example \cite[Example~3.2.4.4]{lurie-ha} (and possibly a good explanation at \cite{nardin-mathoverflow}).) \\
Monoids for this monoidal structure can be identified with functors valued in monoid objects in $\ct$
$$\cj \to \text{CAlg}(\ct)$$
i.e.\ functors $F \colon \cj \to \ct$ equipped with compatible maps
$$F(V) \otimes_\ct F(V) \to F(V)$$
as well as
$$1_\ct \to F(V)$$
for any $V \in \cj$, where $1_\ct$ denotes the monoidal unit in $\ct$.

\paragraph*{The results of Hahn-Yuan on the pointwise monoidal structure in Orthogonal Calculus} $ $\\
The result (and approach) we are describing here is for functors with values in spectra, so $\ct = \spc$, and $\fun(\cj,\spc)$ is understood to be equipped with the pointwise monoidal structure. It is \cite[Theorem~4.10]{hahn-yuan}.

\begin{definition}
The category of towers in $\fun(\cj,\ct)$ is 
$$\fun(\mathbb{Z}_{\geq 0}^{op}, \fun(\cj,\ct))$$
 On this category, there is a convolution monoidal structure, given by Day convolution, i.e.\
$$(X \conv Y)_n = \colim_{p+q \geq n} X_p \otimes Y_q \simeq \oplus_{p+q=n} X_p \otimes_{\text{pointwise}} Y_q $$
\end{definition}
\begin{proposition}[\cite{hahn-yuan}, Construction 4.4.]
The polynomial approximations $T_n$ ($n \in \mathbb{N}$) assemble into a functor
$$\Tow \colon \ct^\cj \to \fun(\mathbb{Z}_{\geq 0}^{op}, \fun(\cj,\ct))$$
such that $\Tow(F)(n)=T_nF$.
\end{proposition}
\begin{proposition}[\cite{hahn-yuan}, Lemma 4.5]
$\Tow$ is oplax symmetric monoidal.
\end{proposition}
Recall that an oplax structure on a functor $F \colon \cC \to \cD$ consists of structure maps
$$F(A \otimes_\cC B) \to F(A) \otimes_\cD F(B) $$
In order for these oplax structure maps to be equivalences (and thus to obtain a symmetric monoidal structure from this oplax symmetric monoidal structure) they restrict to the following subclass of functors.
\begin{definition}[\cite{hahn-yuan}, Definition 4.6.]
Let $F \in \spc^\cj$ be a functor. Call $F$ rapidly convergent if $F$ takes values in connective spectra and there exist real numbers $c,\alpha >0$ such that the natural map $F(W) \to T_nF(W)$ is $(\alpha n)\dim W -c$ - connected. Denote by $\spc^\cj_{conv}$ the category of rapidly convergent functors.
\end{definition}

\begin{theorem}[\cite{hahn-yuan}, Theorem 4.10]
The functor 
$$\Tow \colon \spc^\cj_{conv} \to \fun(\mathbb{Z}_{\geq 0}^{op}, \spc^\cj_{conv})$$
is symmetric monoidal. 
\end{theorem}
\begin{remark}
Their result should also be true in Goodwillie calculus (even without a rapid convergence hypothesis).
\end{remark}

\subsection{The Day Convolution Monoidal Structure}
\label{day_convolution_exposition}
For nice symmetric monoidal categories 
$\cC$ and $\cD$ ($\cD$ cocomplete) the functor category $\fun(\cC,\cD)$ can be equipped with the {\it Day convolution product}. This product is given on functors $F,G$ by
$$(F \conv G)(C) := \colim_{C_1 \otimes_\cC C_2 \to C} F(C_1) \otimes_\cD G(C_2)$$

Day convolution goes back to Brian Day in the 1-categorical setting (\cite{day-article, day-thesis}) and has been made available in the $\infty$-categorical world by Saul Glasman (\cite{glasman}) and later in greater generality by Jacob Lurie (\cite[Section~2.2.6]{lurie-ha}). \\

Another description of $F \conv G$ is as the left Kan extension of $\otimes_\cD \circ (F \times G)$ along $\otimes_\cC$ as shown in the diagram

\begin{center}
\begin{tikzcd}[column sep = 5em, row sep = 5em]
\cC \times \cC \arrow[rr, "F \times G"] \arrow[dd, "\otimes_\cC"] \arrow[rrdd, "\otimes_\cD \circ (F \times G)", ""'{name=U}] &  & \cD \times \cD \arrow[dd, "\otimes_\cD"] \\
                                                                                                                 &  &                                          \\
\cC \arrow[rr, "F \conv G := Lan_{(\otimes_\cD \circ (F \times G))}\otimes_\cC"]                          
\arrow[Rightarrow, from=U, shorten > = 2em, shorten < = 1em] &  & \cD                                     
\end{tikzcd}
\end{center}

where the upper triangle commutes and the lower triangles commutes up to a universal natural transformation 
$$\otimes_\cD \circ (F \times G) \Rightarrow (Lan_{(\otimes_\cD \circ (F \times G))} \otimes_\cC) \circ \otimes_\cC$$

A key property of Day convolution is that monoid objects in this monoidal category are precisely lax symmetric monoidal functors, i.e.\ functors $F \colon C \to D$ equipped with structure maps of the form
$$F(C_1) \otimes_D F(C_2) \to F(C_1 \otimes_C C_2)$$
as well as unit maps
$$1_\cD \to F(1_\cC) $$

The aforementioned fact that monoids for the Day convolution structure are precisely lax symmetric monoidal functors is our original reason for considering it. Many functors of interest such as
\begin{itemize}
\item $V \mapsto S^V$
\item $V \mapsto O(V)$
\item $V \mapsto \BO(V)$
\item $V \mapsto \BTOP(V)$
\end{itemize}
are lax symmetric monoidal\footnote{
The required maps $O(V) \times O(W) \to O(V \oplus W)$ are given by block sum of matrices. The maps for the other functors are analogous or derived from this one.
}, and one can ask if the same holds for their derivatives and Taylor approximations. This would be implied by the statement that
$T_n \colon \fun(\cj,\ct) \to \fun(\cj,\ct)$
\enquote{restricts} to a functor
\[\begin{tikzcd}
	{\fun(\cj,\ct)} && {\fun(\cj,\ct)} \\
	{\fun^{\text{lax symm.\ monoidal}}(\cj,\ct)} && {\fun^{\text{lax symm.\ monoidal}}(\cj,\ct)} \\
	{\text{CAlg}(\fun(\cj,\ct),\conv)} && {\text{CAlg}(\fun(\cj,\ct),\conv)}
	\arrow["{T_n}", from=1-1, to=1-3]
	\arrow[from=2-1, to=1-1]
	\arrow[from=2-3, to=1-3]
	\arrow[dotted, from=2-1, to=2-3]
	\arrow[dotted, from=3-1, to=3-3]
	\arrow["\cong"{description}, no head, from=2-1, to=3-1]
	\arrow["\cong"{description}, no head, from=2-3, to=3-3]
\end{tikzcd}\]

The \enquote{best} explanation for a functor between the categories of commutative algebra objects in symmetric monoidal categories to exist is a lax symmetric monoidal functor between the symmetric monoidal categories (and not just between the monoid objects). This motivates the following central question, which we will answer positively.
\begin{question}
Are the functors
$$T_n \colon \fun(\cj,\ct) \to \fun(\cj,\ct)$$
lax symmetric monoidal for the Day convolution monoidal structure? 
\end{question}
\begin{question}
What can be said about the derivative spectra $\Theta^nF$ of a lax symmetric monoidal functor $F$?
\end{question}

In \cref{orthogonal_monoidality_theorem}, we will answer the first question positively, i.e.\ show that the $T_n$ are indeed lax symmetric monoidal. In particular, the Taylor approximations of lax symmetric monoidal functors are themselves lax symmetric monoidal. \\
Furthermore, in \cref{monoidal_oc_spectra}, we examine the implications for derivative spectra. The main result is \cref{final_derivative_spectra_result}, which states the derivative spectra of any lax symmetric monoidal functor $F \colon \cj \to \csp$ admit maps
$$\Theta^nF \otimes \Theta^nF \to \Theta^nF \otimes D_{O(n)}$$
where $D_{O(n)} \simeq S^{\text{Ad}_n}$ is the dualising spectrum of the topological group $O(n)$.

\chapter{Other Work on Orthogonal Calculus}
\label{chapter_survey}
\chaptermark{Other work}

This chapter is a collection of several versions of, calculations with and statements about Orthogonal Calculus that are in the literature. We hope that it might serve as a useful starting point for the reader who wishes to delve deeper into the subject. 

\section{Formalisms for Orthogonal Calculus}
\sectionmark{Different Formalisms}

Orthogonal Calculus has been developed and used in different formalisms. We briefly mention some of them and their differences.

\subsection{The Original Version (Weiss)}
\label{subsection_the_original_version_(Weiss)}
In Weiss's original paper \cite{weiss-oc}, the category $\cj$ of finite-dimensional real inner-product spaces and linear isometric embeddings is a $1$-category enriched in (topological) spaces, and the functors $F$ which are studied are supposed to be continuous in the sense that the evaluation maps
$$\mor(V,W) \times F(V) \to F(W)$$ 
are continuous. If the category of spaces occuring as the target of the functors is understood to be a convenient category of topological space in the sense that it is cartesian closed (e.g.\ CGWH spaces), then one can equivalently ask for the maps
$$\mor(V,W) \to \mor(F(V),F(W))$$
to be continuous, hence the setting is that of categories enriched in nice topological spaces and enriched functors between them.

In this framework, one has to be careful about the meaning of
$$\lim_{0 \neq U \subseteq \mathbb{R}^{n+1}}F(U \oplus V)$$
It is a $\text{holim}$, and extra care has to be taken as the holim is not indexed over a set or a finite diagram, but over a space. See also \cite[Comments~to~Definition~5.1]{weiss-oc}.

In defining the derivatives, an important role is played by categories called $\cj_n$. These are categories enriched in pointed spaces with the same objects as $\cj$ and
$$\mor_{\cj_n}(V,W) = \Th(\gamma_n(V,W))$$
where $\gamma_n \to \mor_\cj(V,W)$ is the $n$-th tensor power of the orthogonal complement bundle on $\mor_\cj(V,W)$ and $\Th$ denotes the Thom space. Starting with a functor $F$ from $\cj$ to spaces, one obtains a functor $F^{(n)}$ on $\cj_n$ by left Kan extension. $F^{(n)}$ might be called the unstable $n$-th derivative of $F$. The derivative spectrum $\Theta^n F$ is related to $F^{(n)}$ by $(\Theta^n F)_{nk} = F^{(n)}(\mathbb{R}^k)$. An inductive construction using fibre sequences for $F^{(n)}$ is also available (see e.g.\ \cite{weiss-oc}, p.2), but obscures the $O(n)$-action. 

A slight difference in the (statement of) properties of Orthogonal Calculus is that in this enriched, $1$-categorical setup, the subcategory of polynomial functors $\poly^{\leq n}(\cj,\cs) \subset \fun(\cj,\cs)$ is not quite a reflective subcategory. (See \cite[Chapter~6]{weiss-oc} for details.)

\subsection{Model Categories (Barnes-Oman)}

Orthogonal Calculus has been re-developed in the language of model categories by Barnes and Oman (\cite{barnes-oman}). In their description, there are several model structures on the category of functors from $\cj$ to spaces. More concretely, there is a model category $n$-poly-$\mathcal{E}_0$ which is a Bousfield localisation of the standard projective model structure on $\fun(\cj,\cs)$ such that the fibrant objects of $n$-poly-$\mathcal{E}_0$ are precisely the $n$-polynomial functors. This means that the $n$-polynomial approximation functor is the fibrant replacement functor in $n$-poly-$\mathcal{E}_0$. 

An application of their work is given in (\cite[Section~11]{barnes-oman}) where they provide a rigorous construction of Orthogonal Calculus for functors with values in (orthogonal) spectra. Similar model-categorical techniques also play an important role in the later work of Taggart (e.g.\ \cite{taggart-unitary, taggart-reality}) who constructed Unitary Calculus as well as Calculus with Reality, which are functor calculi similar to Orthogonal Calculus for functors defined on the category of complex vector spaces resp.\ vector spaces with Reality.

\subsection{$\infty$-categorical Orthogonal Calculus}

It is possible to express the methods and results of Orthogonal Calculus in the framework of $\infty$-categories, and we will use this version in much of this thesis. As we explain it in detail in \cref{infinity_oc}, we don't say much here.

Many basic statements and properties can be very easily formulated in the language of $\infty$-categories. As an example, the discussion in (\cite[Chapter~6]{weiss-oc}) mentions the problem that $\poly^{\leq n}(\cj,\cs) \subseteq \fun(\cj,\cs)$ is not quite a reflective subcategories. In $\infty$-land, this problem disappears, and we get

\begin{proposition}[\cref{infty_oc_taylor_tower}]
 $ $ \\
Let $\cj$ be the $\infty$-category of finite dimensional $\mathbb{R}$-vector spaces and linear isometric maps\footnote{
A priori, this describes a category enriched in (good) topological spaces. Applying $\text{Sing}$ yields a category enriched in Kan complexes, from which the homotopy coherent nerve produces an $\infty$-category. 
}.
Then for each $n \geq 0$ there are the subcategories $\poly^{\leq n} \subset \fun(\cj,\cs)$ of $n$-polynomial functors. Their inclusions admit left-exact left-adjoints $T_n$
\[\begin{tikzcd}
	{\poly^{\leq n}} & {\fun(\cj,\cs)}
	\arrow[""{name=0, anchor=center, inner sep=0}, "\iota"', curve={height=12pt}, hook, from=1-1, to=1-2]
	\arrow[""{name=1, anchor=center, inner sep=0}, "{T_n}"', curve={height=12pt}, from=1-2, to=1-1]
	\arrow["\dashv"{anchor=center, rotate=-90}, draw=none, from=1, to=0]
\end{tikzcd}\]
\end{proposition}

A similar approach to Goodwillie Calculus has been given by Lurie in \cite[Chapter~6]{lurie-ha}, and our approach is indeed inspired by his.

This perspective has also been used and partially described by Hahn and Yuan \cite{hahn-yuan}.

\paragraph{An Alternative Topos-theoretic Construction} $ $ \\

A completely alternative construction using topos-theoretic methods has been announced by Anel, Biedermann, Finster and Joyal (\cite{abfj-generalised_goodwillie_towers}). They associate a \textit{Generalised Goodwillie tower} to any left-exact localisation of an $\infty$-topos, and claim that Orthogonal Calculus can be recovered by considering the example

$$\fun(\cj,\cs) \to \cs $$
$$F \mapsto F(\mathbb{R}^\infty) = \colim_n F(\mathbb{R}^n)$$

See also \cref{topos_theoretical_remarks} for some further comments on this perspective.

\section{Some Known Calculations in Orthogonal Calculus}
\sectionmark{Known Calculations}

In this section, we collect some calculations and applications of Orthogonal Calculus from the literature.

\subsection{Weiss's Original Article}

The following are some of the examples given in the original article (\cite{weiss-oc}). 

\begin{example}
[\cite{weiss-oc}, Introduction, Example 2.7, Example 10.6]
For 
$$E(V)=\BO(V)$$
the first three derivative spectra are given by
$$\Theta^1 E \simeq \mathbb{S}$$
$$\Theta^2 E \simeq \Omega \mathbb{S}$$
$$\Theta^3 E \simeq \Omega^2 \mathbb{S}/3$$
\end{example}
\begin{example}
[\cite{weiss-oc}, Introduction]
For
$$E(V) = \BTOP(V)$$
the first derivative spectrum is
$$\Theta^1 E \simeq A(*)$$
where $A(*)$ is Waldhausens A-theory spectrum of a point, i.e.\ $A(*) \simeq K(\mathbb{S})$.
\begin{remark}
Recall that for a manifold $M$, one possible description of $A(M)$
is
$$A(M) = K(\mathbb{S}[\Omega M])$$
where $\Omega M$ is the loop space of $M$ (as a monoid), $\mathbb{S}[\Omega M]$ is the spherical group ring of that monoid, and $K(\mathbb{S}[\Omega M])$ is its $K$-theory spectrum  (\cite{waldhausen-algebraic-k-theory-of-spaces}).
\end{remark}
\end{example}

\begin{example}
[\cite{weiss-oc}, Example 3.2.]
Let $M$ be a smooth compact manifold, potentially with boundary. Consider
$$F(V)=B \text{Diff}^b_\partial(M \times V)$$
where 
$$\text{Diff}^b_\partial (M \times V)$$
is the group of diffeomorphisms of $M \times V$ which are the identity on $\partial M \times V$ and which are bounded in the sense that there is a constant $c > 0$ such that $||p_2(f(m,v))-v|| \leq c$ for all $(m,v) \in M \times V$.

Then 
$$\Theta^1 F \simeq \Omega \text{Wh}^\text{Diff}(M)$$
where $\text{Wh}^\text{Diff}$ is the smooth Whitehead spectrum of $M$. This is related to the $A$-theory spectrum $A(M)$ by $A(M) \simeq \Omega^\infty S^\infty (M_+) \times \text{Wh}^\text{Diff}(M)$ (see \cite{waldhausen-algebraic-k-theory-of-spaces}, Theorem 2, Addendum 4).
\end{example}

\subsection{Arone's Work on BO and BU}

Recall, as mentioned in \cref{unitary-calculus}, that Orthogonal Calculus works just as well in the setting of complex vector spaces as it does in real vector spaces.

In his paper \cite{arone-bo-bu}, Arone studies the derivative spectra of the functors
$$V \mapsto \BO(V)$$
and
$$V \mapsto \BU(V)$$
We outline his results briefly. Let $\mathbb{F} \in \left\lbrace\mathbb{R},\mathbb{C}\right\rbrace$.

\begin{theorem}[\cite{arone-bo-bu}, Theorem 2]
For $n \geq 1$, the $n$-th derivative of $V \mapsto BAut(V)$ is the spectrum
$$\text{Map}_*(L_n,\Sigma^\infty S^{Ad_n})$$
where $L_n$ is the unreduced suspension of the geometric realisation of the category of non-trivial direct-sum decompositions of $\mathbb{F}^n$ and $Ad_n$ is the adjoint representation of $Aut(n)$ (i.e.\ $O(n)$ resp.\ $U(n)$).
\end{theorem}

Note that $\Sigma^\infty S^{\text{Ad}_n} \simeq D_{O(n)}$ is the dualising spectrum which also appears in our result \cref{final_derivative_spectra_result}.

He then goes on to study the (non-equivariant, or only partially equivariant) homotopy type of these spectra:

\begin{theorem}[\cite{arone-bo-bu}, Theorem 3]
There is an $O(n-1)$-equivariant equivalence
$$\text{Map}_*(L_n^\mathbb{R},\Sigma^\infty S^{Ad_n}) \simeq \text{Map}_*(S^1 \wedge K_n,\Sigma^\infty S^0) \wedge_{\Sigma_n} O(n-1)_+ $$
and a $U(n-1)$ equivariant equivalence
$$\text{Map}_*(L_n^\mathbb{C},\Sigma^\infty S^{Ad_n^\mathbb{C}}) \simeq \text{Map}_*(S^1 \wedge K_n,\Sigma^\infty S^n) \wedge_{\Sigma_n} U(n-1)_+ $$
where $K_n$ denotes the unreduced suspension of the category of non-trivial partitions of $\left\lbrace 1, \ldots , n \right\rbrace$.
\end{theorem}

The key idea here is to use the fibre sequences
$$S^V \to \BO(V) \to \BO(V \oplus \mathbb{R})$$
respectively
$$\Sigma S^V \to \BU(V) \to \BU(V \oplus \mathbb{C})$$
and the close comparison between the Goodwillie tower of the identity functor on spaces and the orthogonal or unitary derivatives of $S^V$, coming from the comparison functor
\begin{align*}
\cj &\to \cs \\
V &\mapsto S^V
\end{align*}
inducing a functor (by precomposition)
$$\fun(\cs,\cs) \to \fun(\cj,\cs)$$
(See also \cref{comparison_with_goodwillie_calculus} for more on the comparison functor).

He also calculates the mod $p$  - cohomology of the derivative spectra (\cite[p.456]{arone-bo-bu}), which turns out to be a free module over a certain subalgebra $A_k$ of the Steenrod algebra $A$. Strikingly, the $n$-th derivative spectra of $U(n)$ resp.\ $O(n)$, $M_n^\mathbb{C}$, are contractible for all $n$ which are not prime powers, and $M_n^\mathbb{R}$ is contractible for all $n$ which are not prime powers of twice a prime power. ($M_n$ denotes the $n$-th derivative spectrum of $\text{Aut}(n)$.)

\subsection{{$\Theta^2 \text{btop}$} (Krannich-Randal-Williams)}

Three very important example functors in Orthogonal Calculus are 

\begin{align*}
\text{BO}(-) \colon V &\mapsto \BO(V) \\
\text{BTOP}(-) \colon V &\mapsto \BTOP(V) = \text{B}\text{Homeo}(V,V) \\
\text{BG}(-) \colon V &\mapsto \text{BG}(V)
\end{align*}
where $G(V)$ is the grouplike monoid of self-homotopy equivalences of $V$. 

We already mentioned that the first derivative of $\text{BTOP}(-)$ is $A(*) \simeq K(\mathbb{S})$. Already the second one remains somewhat mysterious. The following partial answer has been given by Krannich and Randal-Williams (\cite{krannich-randal-williams}).

\begin{theorem}[\cite{krannich-randal-williams}, Theorem 9.9.]
$\Theta^2 \text{BTOP}(-) \to \Theta^2 \text{BG}(-)$ is a rational equivalence. The latter was calculated by (\cite[Prop.~3.5]{reis-weiss}) to be rationally equivalent to $Map(S^1_+, \mathbb{S}^{-1})$
\end{theorem}

\subsection{Pontryagin Classes etc. (Reis-Weiss et al.)}

In \cite{reis-weiss}, Reis and Weiss tried to use Orthogonal Calculus to see that that the square of the euler class is the Pontryagin class, not just in $H^*(BSO(2m))$ but also in $BSTOP(2m)$. 

Later, Weiss showed in \cite{weiss-dalian} that this is not true. 

Even later, Galatius-Randal-Williams \cite{galatius-randal-williams_pontryagin} showed that the topological Pontryagin classes are in fact algebraically independent in $BTOP(2m)$.

\begin{corollary}
$Bo \to Btop$ does not admit a rational left inverse.
\end{corollary}
\begin{proof}
According to Reis-Weiss, the existence of such a left inverse would be equivalent to $e^2 = p_{\frac{n}{2}} \in H^{2n}(BSTOP(n);\mathbb{Q})$. But this is not true, so there is no such splitting.
\end{proof}

\section{Analyticity Conditions}

Recall that a functor $F \colon \cj \to \cs$ is called \textit{analytic} if it can be recovered from its Taylor tower, i.e.\ if
$$F \simeq \lim_n T_nF$$
Since limits of functors are computed pointwise, this implies that
$$F(V) \simeq \lim_n T_nF(V)$$
for any $V \in \cj$.

Unfortunately, this notion of analyticity is quite hard to check in practice. However, one can often show weaker versions, as we will discuss shortly.  

\subsection{Barnes-Eldred}
\begin{definition}[\cite{barnes-eldred-comparison}, Def.\ 2.25]
A functor $F$ is weakly $\rho$-analytic if and only if 
$$F(V) \to T_\infty F(V)$$ 
is an equivalence for all $V$ of dimension $\geq \rho$.
\end{definition}

\begin{proposition}[\cite{barnes-eldred-comparison}, Thm. 4.1]
$\text{BO}(-)$ is weakly $2$-analytic. $\text{BU}(-)$ is weakly $1$-analytic.
\end{proposition}
\begin{proof}
Their proof uses their comparison to Goodwillie Calculus and the fact that the first unstable derivative of $\text{BO}(-)$ is $V \mapsto S^V$.
\end{proof}

\subsection{Taggart}
\begin{definition}[\cite{taggart-thesis} Definition 3.4.6.]
A functor $F$ is weakly $(\rho,n)$-polynomial if the map
$$\eta \colon F(U) \to T_nF(U)$$
is an agreement of order $n$ for all $U$ with $\dim U \geq \rho$. A functor is weakly polynomial if there exists $\rho$ such that it is weakly $(\rho,n)$-polynomial for all $n \geq  0$.
\end{definition}

where

\begin{definition}
A map $p \colon F \to G$ is an order $n$ agreement if there is some $\rho \in \mathbb{N}$ and $b \in \mathbb{Z}$ such that $p_U \colon F(U) \to G(U)$ is $((n+1) \dim U -b)$-connected for all $U$ satisfying $\dim U \geq \rho$. We say that $F$ agrees with $G$ to order $n$ if there is an order $n$ agreement between them.
\end{definition}

The point of these notions is the following connection to convergence of the Taylor tower

\begin{lemma}[\cite{taggart-thesis} Lemma 3.4.5.]
If for all $n \geq 0$ the map $F \to T_nF$ is an order $n$ agreement then the Taylor tower of $F$ converges to $F(V)$ at $V$ for $\dim V \geq \rho$.
\end{lemma}

\begin{example}[{\cite[Examples~3.4.9]{taggart-thesis}}, {\cite[Section~9]{taggart-unitary}}]
The following functors are weakly polynomial:
\begin{itemize}
\item $V \mapsto S^V$
\item $\mapsto BO(V)$ ($\rho = 2$)
\item $V \mapsto BU(V)$ (in unitary calculus)($\rho = 1$, \cite{taggart-unitary} Example 9.16.)
\item $V \mapsto O(V)$
\item Representable functors $V \mapsto \text{Map}_\cj(U,V)$ (\cite{taggart-unitary}, Theorem 9.17.)
\end{itemize}

\end{example}

\section{Various}

\subsection{Unitary Calculus}

\begin{remark}[Unitary Calculus]
\label{unitary-calculus}
There is a complex analogue of Orthogonal Calculus called \enquote{Unitary Calculus}. In complete analogy with Orthogonal Calculus, it studies functors 
$$\cj \to \cs$$
where now, $\cj$ denotes the category of finite-dimensional vector spaces over $\mathbb{C}$ with a positive definite inner product.

That this analogue exists has been known from the beginning (e.g. \cite[Example~10.4]{weiss-oc}). A relatively modern treatment (using model-categorical techniques) has recently been given by Taggart in \cite{taggart-unitary} and subsequently been used in his \enquote{Unitary Calculus with Reality} (\cite{taggart-reality}), a unitary version which takes into account the complex conjugation action on the complex vector spaces and is able to specialise to both Orthogonal and Unitary Calculus.
\end{remark}

Other places in the literature where Unitary Calculus made an appearance include:
\begin{itemize}
\item Arone (\cite{arone-bo-bu}) calculating the derivatives of the orthogonal functor $V \mapsto BO(V)$ as well the unitary functor $V \mapsto BU(V)$.
\item Hahn and Yuan (\cite{hahn-yuan}) using unitary calculus of the functor (valued in spectra) $V \mapsto \Sigma^\infty_+ \Omega \cj(V,V \oplus W)$ to study the Mitchell-Richter filtration/splitting on $\Sigma_+^\infty \Omega SU(n)$.
\end{itemize}

\subsection{Comparison with Goodwillie Calculus}
\label{comparison_with_goodwillie_calculus}
Goodwillie Calculus and Orthogonal Calculus are closely related. Some explicit statements are

\begin{theorem}[Barnes-Eldred {\cite[Prop.~3.2]{barnes-eldred-comparison}}]
Precomposition with the one-point compactification functor $\cj \to \cs$ induces a functor
$$\cs^\cs \to \cs^\cj$$
which maps $n$-excisive functors (in the sense of Goodwillie) to $n$-polynomial functors (in the sense of Weiss).
\end{theorem}

and

\begin{theorem}[Barnes-Eldred {\cite[Section~5]{barnes-eldred-comparison}}, in particular commutativity of the squares 1,2,3] $ $ \\
Let $F_{GW} \colon \cs \to \cs$ be some functor and let $F_W \colon \cj \to \cs$ be its composition with the one-point compactification functor $\cj \to \cs$. Then the $n$-th derivative spectrum of $F_W$ (in Orthogonal Calculus) is obtained from the $n$-th derivative spectrum of $F_{GW}$ (in Goodwillie Calculus) by induction along $\Sigma_n \to O(n)$, i.e.
$$\Theta^n (F_W) \simeq \text{ind}_{\Sigma_n}^{O(n)} \Theta^n_{GW} (F_{GW}) \simeq O(n)_+ \wedge_{\Sigma_n} \Theta^n_{GW} (F_{GW})$$
where $\Theta^n(F_W)$ denotes the derivative spectrum in Orthogonal Calculus of $F_W$ and $\Theta^n_{GW}$ denotes the derivative spectrum of $F_{GW}$ in the sense of Goodwillie calculus.
\end{theorem}

\chapter{$\infty$-Categories for Orthogonal Calculus} 
\label{infinity_oc}
When Orthogonal Calculus was first introduced by Weiss in 1995 (\cite{weiss-oc}), the theory of quasicategories / $\infty$-categories as developed later by Lurie (\cite{lurie-htt, lurie-ha}) was not yet available. As we will make essential use of $\infty$-categorical constructions later, we describe the essential features of Orthogonal Calculus in the language of $\infty$-categories in this chapter. The key justification for the passage to $\infty$-land\footnote{
also referred to as $\infty$-land, see e.g.\ Thomas Nikolaus' talk at ICM 2022 (\cite{nikolaus-icm}).}
 will be \cref{infty_oc_taylor_tower}.

This perspective on Orthogonal Calculus has not found widespread use in the literature yet. It has appeared in a paper of Hahn-Yuan (\cite{hahn-yuan}), and a new construction of \enquote{Generalised Goodwillie towers}, which includes Orthogonal Calculus, has been announced by Anel, Biedermann, FInster and Joyal (\cite{abfj-generalised_goodwillie_towers}).

\section{Basic Definitions of Relevant Categories and Functors}
\sectionmark{Basic Definitions}

\begin{definition}
\label{infty_oc-definition-cjtop}
Let $\cjtop$ be the category enriched in topological spaces whose objects are finite-dimensional real inner product spaces $V$ with morphism spaces $\mor(V,W)$ given by the spaces of linear isometric embeddings
$$\mor(V,W) := \{f \colon V \to W \mid f \text{ linear and } \langle v_1,v_2 \rangle_V = \langle f(v_1),f(v_2) \rangle_W\} $$
with the obvious composition of maps.
\end{definition}

\begin{remark}
Choosing identifications $V \cong \mathbb{R}^k$ und $W \cong \mathbb{R}^n$ leads to an identification 
$$\mor_\cjtop(V,W) \cong V_k(\mathbb{R}^n) \cong O(n)/O(n-k)$$
where $V_k(\mathbb{R}^n)$ is the Stiefel manifold of orthonormal $k$-frames in $\mathbb{R}^n$ and $O(n)/O(n-k)$ is the space of cosets of $O(n-k)$ in $O(n)$.
\end{remark}

\begin{definition}
\label{infty_oc-definition-cj}
Let $\cj$ denote the $\infty$-category corresponding to the enriched category $\cj_{\text{top}}$ of \cref{infty_oc-definition-cjtop}. More precisely, $\cj$ is obtained by applying the homotopy-coherent nerve construction (see e.g.\ \cite[Def.\ 1.1.5.5, Cor.\ 1.1.5.12]{lurie-htt}) to the category enriched in simplicial sets (even in Kan complexes) $\cjenr$, which itself arises by taking nerves of the hom-spaces of $\cjtop$. {(This composition is also called the \enquote{topological nerve}.)}
\end{definition}

\paragraph{Functors on $\cj$} $ $\\
Orthogonal Calculus primarily studies functors from $\cj$ to 
\begin{itemize}
\item $\cs$, the category of spaces.
\item $\cs_*$, the category of pointed spaces.
\item $\spc$, the category of spectra. (\cite[Section~11]{barnes-oman})
\end{itemize}

\begin{remark}
Orthogonal Calculus can also be done for functors with values in localisations of $\cs$, as is shown in \cite{taggart-local}. As far as we know, the question \enquote{What is the highest generality, especially concerning the target category, for which Orthogonal Calculus can be done?} is still open. We note that Lurie has a notion of \enquote{differentiable $\infty$-category} (\cite{lurie-ha}, 6.1.1.6), which might be a candidate answer to this question.
\end{remark}

Some important examples of functors on $\cj$ are

\begin{example}
\label{example-functors}
\begin{enumerate}[label=\alph*)] $ $
\item 
$
\cj \to \cs \\
V \mapsto S^V \cong V^c
$
\item
$
\cj \to \cs \\
V \mapsto O(V)
$
\item 
$
\cj \to \cs \\
V \mapsto \B O(V)
$
\item 
$
\cj \to \cs \\
V \mapsto \B TOP(V)
$
\item 
$
\cj \to \cs \\
V \mapsto \B G(V)
$
\item For fixed $W \in \cj$, 

$$F_W \colon \cj \to \spc$$
$$V \mapsto \Sigma_+^\infty \Omega (\mor_\cj(W, W \oplus V))$$

\end{enumerate} 
Note that these in fact define functors to pointed spaces as well.
\end{example}

\begin{proof}[Proof of functoriality]
Since $$\cj \simeq N_{hcoh} (\text{Sing}_* (\cj_{top}))$$
where
$$\text{Sing}_* \colon \Cat_{top} \to \Cat_{Kan}$$
is the functor which associates to a category enriched in topological spaces the corresponding category enriched in Kan complexes
and
$$N_{hcoh} \colon \Cat_{Kan} \to \Cat_\infty$$
is the homotopy-coherent nerve which associates to a category enriched in Kan complexes the corresponding $\infty$-category,
it is sufficient to describe the corresponding topologically enriched functors from $\cj_{top} \to \text{Top}$ (by functoriality of $\text{Sing}_*$ and $N_{hcoh}$).

To do that, let's examine the case $V \mapsto O(V)$ first. Let $f \colon V \to W$ be a morphism in $\cj_{top}$. Then any $\varphi \in O(V)$ can be extended by the identity on the orthogonal complement of $f(V)$ in $W$, i.e.\
$$O(V) \to O(V \oplus W)$$
$$\varphi \to \begin{pmatrix}
\varphi & 0 \\
0 & id_{W \ominus f(V)}
\end{pmatrix}
$$

The continuity is obvious.
The checks required for the other examples work similarly.
\end{proof}

\begin{remark}
\label{remark_oc_over_more_general_fields}
We observe that the functoriality checks only makes use of the definition of $\cj$, in particular the isometricity condition in the mapping spaces and the inner product, insofar as the (natural) orthogonal complements are used to extend automorphisms on subspaces to the whole vector space. One can wonder whether this is a starting point for a generalisation for Orthogonal Calculus. See also \cref{question_oc_over_more_general_fields}.
\end{remark}

\section{Polynomial Functors}

The central notion to functor calculus is that of polynomial functors. 

Recall from ordinary calculus that for a polynomial function $f$ of degree at most $n$, the interpolation polynomial at $n+1$ distinct points is $f$ itself.

In Orthogonal Calculus, a similar interpolation property is turned into a definition.

\begin{definition}
\label{polynomial_of_degree_n}
	A functor $F \colon \cj \to \cs$ is called \textbf{polynomial of degree $\leq n$} or simply\textbf{ $n$-polynomial} if the canonical natural transformation
	$$F(-) \to \lim_{0 \neq U \subseteq \mathbb{R}^{n+1} \text{ linear subspace}}F(- \oplus U)$$
is an equivalence.	
	
	A functor $F \colon \cj \to \csp$ is \textbf{polynomial of degree $\leq n$} if its composition with the forgetful map $\mathcal{S}_* \to \mathcal{S}$ is. 
	
We denote by $\Poly^{\leq n}(\cj,\cs) \subseteq \ce$ and $\Poly^{\leq n}(\cj,\csp)$ the full subcategories of functors which are polynomial of degree $\leq n$.
\end{definition}

\begin{remark}
The limit
$$ \lim_{0 \neq U \subseteq \mathbb{R}^{n+1} \text{ linear subspace}}F(- \oplus U)$$
appearing in the previous definition is formed over the slice category in $\cj$ over $\mathbb{R}^{n+1}$. Note that, from a classical viewpoint, this is not simply the poset of sub-vector spaces of $\mathbb{R}^{n+1}$, but takes into account the topology. 
\end{remark}

\begin{example}
Examining \cref{polynomial_of_degree_n} for $n=0$, we see that a $0$-polynomial functor $F$ is equivalent to the constant functor with value $F(0)$.
\end{example}

\begin{example}
Let $n=1$. Then the indexing category $0 \neq U \subseteq \mathbb{R}^{2}$ takes the shape of a cone over $\mathbb{R}P^1$, as there is an $\mathbb{R}P^1$ worth of $1$-dimensional subspaces of $\mathbb{R}^2$, and a single $2$-dimensional subspace, into which all of them include. 
\end{example}

\begin{example}
Let $\Theta \in \spc^{\BO(n)}$. Then the functor
$$V \mapsto \Omega^\infty \left(  \left( S^{nV} \wedge \Theta \right)_{hO(n)}\right)$$
is polynomial of degree $n$.
\end{example}

\begin{defprop}
\label{definition_of_T_n}
	Let $F \in \ce$. Define 
		 $$\tau_nF \colon \mathcal{J} \to \cs$$ 
		 $$V \mapsto \lim_{0 \neq U \subseteq \mathbb{R}^{n+1} \text{ linear subspace}} F(V \oplus U)$$
	
	By definition, $F \in \Poly^{\leq n}$ iff the natural map $\rho \colon F \to \tau_nF$ is an equivalence.
	
	Define furthermore $T_nF$ to be the colimit of
	\begin{center}
		\begin{tikzcd}
			F \arrow[r, "\rho"] & \tau_nF \arrow[r,"\tau_n(\rho)"]&  \tau_n^2F \arrow[r,"\tau_n^2(\rho)"]  & \cdots
		\end{tikzcd}
	\end{center}
	The inclusion into the first term induces a natural transformation 
	$$\eta_n \colon \text{id}_\ce \to T_n$$
\end{defprop}

Weiss showed in \cite{weiss-oc}, Theorem 6.3\ that $T_nF$ is always $n$-polynomial.

\paragraph{Polynomial Functors Form A Reflective Subcategory} $ $ \\

{

As Michael Weiss already remarked in the original article, $\Poly^{\leq n}(\cj,\cs) \subset \fun(\cj,\cs)$ \enquote{wants} to be a reflective subcategory, but could not be in the $1$-categorical world. In $\infty$-land, we now show that this problem disappears and $T_n \colon \fun(\cj,\cs) \to \poly^{\leq n}(\cj,\cs)$ is a left adjoint to the inclusion, and $\eta_n$ is the unit of this adjunction.}

\begin{definition}[\cite{lurie-htt}, 5.2.7.9]
A full subcategory $\cC^0$ of an $\infty$-category $\cC$ is called \textbf{reflective subcategory} if the inclusion $\cC^0 \into \cC$ admits a left adjoint.
\end{definition}

Much of the work needed to establish an $\infty$-categorical version of Orthogonal Calculus has already been done in the original article \cite{weiss-oc} (and the erratum \cite{weiss-oc-err}), especially in Theorem 6.3. We repeat it here, noting that we have already adapted the notation.

\begin{theorem}[{\cite[Thm.~6.3]{weiss-oc}}, \cite{weiss-oc-err}]

\label{weiss-oc-6.3.}
For any $n \geq 0$, there exists a functor
$$T_n \colon \ce \to \ce$$
(constructed as in \cref{definition_of_T_n})
taking equivalences to equivalences, and a natural transformation
$$\eta_n \colon id_{\ce} \to T_n$$
with the properties
\begin{enumerate}
\item $T_n(E)$ is polynomial of degree $\leq n$ for all $E \colon \cj \to \cs$
\item If $E$ is already polynomial of degree $\leq n$, then
$$\eta_{n,E} \colon E \to T_nE$$
is an equivalence.
\item For every $E$, the map
$$T_n(\eta_{n,E}) \colon T_nE \to T_nT_nE$$
is an equivalence.
\end{enumerate}
\end{theorem}

This theorem covers much of what one needs to know to use the $\infty$-categorical language for Orthogonal Calculus. 
The reason is that these conditions (as was already noted in \cite[Obs.~6.1]{weiss-oc}) precisely characterise reflective subcategories. In $\infty$-land, this is the following proposition

\begin{proposition}[{\cite[Prop.~5.2.7.4]{lurie-htt}}]
	\label{5274}
Let $\mathcal{C}$ be an $\infty$-category and let $L \colon \mathcal{C} \to \mathcal{C}$ be a functor with essential image $L\mathcal{C} \subseteq \mathcal{C}$. The following conditions are equivalent:
\begin{enumerate}[label=(\alph*)]
\item There exists a functor $F \colon \mathcal{C} \to \mathcal{D}$ with a fully faithful right adjoint $G \colon D \to C$ and an equivalence between $G \circ F$ and $L$.
\item When regarded as a functor from $\mathcal{C}$ to $L\mathcal{C}$, $L$ is left adjoint to the inclusion $L\mathcal{C} \subseteq \mathcal{C}$.
\item There exists a natural transformation 
$$\alpha \colon \mathcal{C} \times \Delta^1 \to \mathcal{C}$$
from $id_\mathcal{C}$ to $L$ such that, for every object $C$ of $\mathcal{C}$, the morphisms 
$$L(\alpha(C)) \colon L(C) \to L(L(C))$$
and 
$$\alpha(LC) \colon L(C) \to L(L(C))$$
are equivalences.
\end{enumerate}
\end{proposition}

\begin{corollary}
\label{infty_oc_taylor_tower}
 $ $ \\
The inclusions of the subcategories 
$$\poly^{\leq n}(\cj,\cs) \subseteq \fun(\cj,\cs)$$
 of $n$-polynomial functors admit left-exact left-adjoints $T_n$

\[\begin{tikzcd}
	{\poly^{\leq n}(\cj,\cs)} && {\fun(\cj,\cs)}
	\arrow[""{name=0, anchor=center, inner sep=0}, "{T_n}"', curve={height=12pt}, from=1-3, to=1-1]
	\arrow[""{name=1, anchor=center, inner sep=0}, "\iota"', curve={height=12pt}, from=1-1, to=1-3]
	\arrow["\dashv"{anchor=center, rotate=-90}, draw=none, from=0, to=1]
\end{tikzcd}\]

\end{corollary}

\begin{proof}
Apply \cref{5274} to the situation 
\begin{center}
$\cC = \fun(\cj,\cs)$, $L=T_n$, $L\cC = \poly^{\leq n}(\cj,\cs)$
\end{center}
 \cref{weiss-oc-6.3.} gives precisely the third condition.
 
For the left-exactness, let $\lim_K F$ be a finite limit. Then
\[\begin{tikzcd}
	{T_n (\lim_K F)(V) } & {\colim \left( \lim_K F(V) \to  \lim_{0 \neq U \subseteq \mathbb{R}^{n+1}} \lim_K F(U \oplus V) \to \ldots \right)} \\
	{\lim_K T_n (F)(V)} & {\lim_K \colim \left( F(V) \to  \lim_{0 \neq U \subseteq \mathbb{R}^{n+1}}  F(U \oplus V) \to \ldots \right)}
	\arrow["\cong", from=1-1, to=1-2]
	\arrow["\cong", from=2-1, to=2-2]
	\arrow["\cong"{description}, from=1-2, to=2-2]
\end{tikzcd}\]
The two right-hand sides agree since $\lim_K$ commutes both with the other involved limits and the filtered colimit (as $\cs$ is a filtered category).
\end{proof}

\begin{remark}
In \cite{hahn-yuan} (esp.\ Remark 4.2.), Hahn-Yuan used the same argument in order to use Orthogonal Calculus in $\infty$-categorical language.
\end{remark}

\begin{remark}
Since a functor from $\cj$ to pointed spaces is polynomial of degree $\leq n$ if and only if the same functor to unpointed spaces is polynomial of degree $\leq n$ (by definition), many statements we prove about unpointed functors apply verbatim pointed functors.
\end{remark}

\begin{lemma}
Any $n$-polynomial functor is also ($n+1$)-polynomial, i.e.\ there are inclusions
\begin{align*}
\poly^{\leq n-1}(\cj,\cs) &\into \poly^{\leq n}(\cj,\cs) \\
\poly^{\leq n-1}(\cj,\csp) &\into \poly^{\leq n}(\cj,\csp)
\end{align*}
They induce factorisations
\[\begin{tikzcd}
	{\fun(\cj,\cs)} & {\poly^{\leq n}(\cj,\cs)} \\
	& {\poly^{\leq n-1}(\cj,\cs)}
	\arrow["{T_n}", from=1-1, to=1-2]
	\arrow["{T_{n-1}}"', from=1-1, to=2-2]
	\arrow[dotted, from=1-2, to=2-2]
\end{tikzcd}\]
and
\[\begin{tikzcd}
	{\fun(\cj,\csp)} & {\poly^{\leq n}(\cj,\csp)} \\
	& {\poly^{\leq n-1}(\cj,\csp)}
	\arrow["{T_n}", from=1-1, to=1-2]
	\arrow["{T_{n-1}}"', from=1-1, to=2-2]
	\arrow[dotted, from=1-2, to=2-2]
\end{tikzcd}\]
\end{lemma}

\begin{proof}
For the inclusions $\poly^{\leq n-1} \into \poly^{\leq n}$, we refer back to \cite{weiss-oc}, Proposition 5.4. The factorisations follows by uniqueness of left adjoints from \cref{infty_oc_taylor_tower}.
\end{proof}

For a functor $F$, the collection of $T_nF$ together with the comparison maps is usually referred to as the \textbf{Taylor tower} (or \enquote{Orthogonal tower}, or \enquote{Weiss tower}) of $F$.

\begin{corollary}
The $n$-polynomial approximation functors assemble into the following tower of categories
\[\begin{tikzcd}
	& {\poly^{\leq n}(\cj,\cs)} \\
	& \vdots \\
	& {\poly^{\leq 1}(\cj,\cs)} \\
	{\fun(\cj,\cs)} & {\poly^0(\cj,\cs)}
	\arrow["{T_0}", from=4-1, to=4-2]
	\arrow["{T_1}", from=4-1, to=3-2]
	\arrow["{T_n}", from=4-1, to=1-2]
	\arrow[from=1-2, to=2-2]
	\arrow[from=2-2, to=3-2]
	\arrow["{{T_1}_{|\poly^{\leq 1}(\cj,\cs)}}", from=3-2, to=4-2]
\end{tikzcd}\]

For any single functor $F$, the units of the adjunctions assemble into the Taylor-Weiss tower
\[\begin{tikzcd}
	&& {T_nF} \\
	&& \vdots \\
	&& {T_1F} \\
	F && {T_0F}
	\arrow["{\eta_0}"', from=4-1, to=4-3]
	\arrow["{\eta_1}"', from=4-1, to=3-3]
	\arrow["{\eta_n}"', from=4-1, to=1-3]
	\arrow[from=1-3, to=2-3]
	\arrow[from=2-3, to=3-3]
	\arrow[from=3-3, to=4-3]
\end{tikzcd}\]
\end{corollary}

\section{Homogeneous Functors and Derivative Spectra}

If one imagines $F \to T_nF$ to be the \enquote{$n$-th Taylor polynomial of $F$}, the natural next question is \begin{center}
\enquote{What is the $n$-th derivative of $F$?}
\end{center}
Sticking to the analogy with real calculus for just a moment, for $f \colon \mathbb{R} \to \mathbb{R}$, we observe
$$(T_nf - T_{n-1}f)(x) = x^n \frac{f^{(n)}(0)}{n!} $$
In functor calculus, the role corresponding to the term on the right-hand side of this equation is played by $n$-homogeneous functors.

\begin{definition}[\cite{weiss-oc}, Definition 7.1.]
\label{definition_of_homogeneous_functors}
A functor $F \colon \cj \to \cs$ or $F \colon \cj \to \csp$ is \textbf{$n$-homogeneous} if it is polynomial of degree $\leq n$ and all values of $T_{n-1}F$ are contractible. \\
Denote by 
\begin{align*}
\Homog(\cj,\cs) &\into \fun(\cj,\cs) \\
\Homog(\cj,\csp) &\into \fun(\cj,\csp)
\end{align*}
the full subcategories on $n$-homogeneous functors. \\
These categories can also be characterised by pullback diagrams of $\infty$-categories
\[\begin{tikzcd}
	{\poly^{\leq n}(\cj,\cs)} & {\Homog(\cj,\cs)} \\
	{\poly^{\leq n-1}(\cj,\cs)} & {*}
	\arrow[from=1-1, to=2-1]
	\arrow[from=2-2, to=2-1]
	\arrow[from=1-2, to=2-2]
	\arrow[from=1-2, to=1-1]
	\arrow["\lrcorner"{anchor=center, pos=0.125, rotate=-90}, draw=none, from=1-2, to=2-1]
\end{tikzcd}\]
and
\[\begin{tikzcd}
	{\poly^{\leq n}(\cj,\csp)} & {\Homog(\cj,\csp)} \\
	{\poly^{\leq n-1}(\cj,\csp)} & {*}
	\arrow[from=1-1, to=2-1]
	\arrow[from=2-2, to=2-1]
	\arrow[from=1-2, to=2-2]
	\arrow[from=1-2, to=1-1]
	\arrow["\lrcorner"{anchor=center, pos=0.125, rotate=-90}, draw=none, from=1-2, to=2-1]
\end{tikzcd}\]

where the maps $* \to \poly^{\leq n}(\cj,\cs)$ and $* \to \poly^{\leq n}(\cj,\csp)$ map to the constant functor with value the one point space, which is $0$-polynomial and thus $n$-polynomial for any $n$.
\end{definition}

In the pointed case, it is easy to associate $n$-homogeneous functors to any $F \colon \cj \to \csp$ in the following way.

\begin{defprop}
\label{definition_of_L_n}
Let $F \colon \cj \to \csp$ and $n \in \mathbb{N}$. For $V \in \cj$, denote by $L_nF(V)$ the fibre in the fibre sequence
$$L_nF(V) \to T_nF(V) \to T_{n-1}F(V)$$
They assemble into a functor $L_nF \colon \cj \to \csp$. This construction defines a functor $L_n \colon \fun(\cj,\csp) \to \fun(\cj,\csp)$.
\end{defprop}

\begin{proposition}
$L_n$ is right adjoint to the inclusion $\Homog(\cj,\csp) \into \poly^{\leq n}(\cj,\csp)$, i.e.

\[\begin{tikzcd}
	{\poly^{\leq n}(\cj,\csp)} && {\Homog(\cj,\csp)}
	\arrow[""{name=0, anchor=center, inner sep=0}, "{L_n}"', curve={height=12pt}, from=1-1, to=1-3]
	\arrow[""{name=1, anchor=center, inner sep=0}, "\iota"', curve={height=12pt}, from=1-3, to=1-1]
	\arrow["\dashv"{anchor=center, rotate=-90}, draw=none, from=1, to=0]
\end{tikzcd}\]
\end{proposition}

\begin{proof}
Let $H \in \Homog(\cj,\csp)$ and $P \in \poly^{\leq n}(\cj,\csp)$. We want to show that
$$\Nat_{\fun(\cj,\csp)}(H,L_nP) \simeq \Nat_{\fun(\cj,\csp)}(H,P)$$
Using the definition and the adjunction of \cref{infty_oc_taylor_tower}, we calculate 
\begin{align*}
\Nat_{\fun(\cj,\csp)}(H,T_{n-1}P)
	&\simeq \Nat_{\fun(\cj,\csp)}(T_{n-1}H,T_{n-1}P) \\
	&\simeq \Nat_{\fun(\cj,\csp)}(*,T_{n-1}P) \\
	&\simeq *
\end{align*}
hence by the fibre sequence
$$L_nP \to T_nP \simeq P \to T_{n-1}P$$
the claim follows.
\end{proof}

We can include the homogeneous layers in the Taylor-Weiss tower

\[\begin{tikzcd}
	&& {T_nF} & {L_nF} \\
	&& \vdots \\
	&& {T_1F} & {L_1F} \\
	F && {T_0F}
	\arrow["{\eta_0}"', from=4-1, to=4-3]
	\arrow["{\eta_1}"', from=4-1, to=3-3]
	\arrow["{\eta_n}"', from=4-1, to=1-3]
	\arrow[from=1-3, to=2-3]
	\arrow[from=2-3, to=3-3]
	\arrow[from=3-3, to=4-3]
	\arrow[from=1-4, to=1-3]
	\arrow[from=3-4, to=3-3]
\end{tikzcd}\]

Phrased differently, any functor $F$ is the successive extension of $n$-homogeneous functors. 

\begin{remark}
Even though \cref{definition_of_homogeneous_functors} defines homogeneous functors to pointed spaces and to spaces in the same way, constructing an $n$-homogeneous functor out of an arbitrary functor $F$ is more complicated in the unpointed case, as a choice of basepoints is needed to take the fibres in \cref{definition_of_L_n}.
 For our purposes, the pointed case will suffice. (See \cref{section_pointed_and_unpointed_functors}.)
\end{remark}

Homogeneous functors to pointed spaces have been classified in the following way.

\begin{defprop}
For any spectrum with $O(n)$-action $\Theta \in \spc^{\BO(n)}$, there is a functor 
\begin{align*}
F \colon \cj &\to \csp \\
V &\mapsto \Omega^\infty \left( \left( S^{V \otimes \mathbb{R}^n} \wedge \Theta \right) \right)_{hO(n)}
\end{align*}
$O(n)$ acts diagonally on the smash product: on $S^{V \otimes \mathbb{R}^n} \simeq {(S^V)}^{\wedge n}$ by the canonical action on $\mathbb{R}^n$ as well as on $\Theta$ by its action on $\Theta$.  \\
This defines a functor 
\begin{align*}
\spc^{\BO(n)} &\to \fun(\cj,\csp) \\
\Theta &\mapsto \left\lbrace V \mapsto  \Omega^\infty \left( \left( S^{V \otimes \mathbb{R}^n} \wedge \Theta \right)_{hO(n)} \right) \right\rbrace
\end{align*}
\end{defprop}

\begin{theorem}[\cite{weiss-oc}, Theorem 7.3.\footnote{Actually, this is the \enquote{more precise version of theorem 7.3.} that Weiss mentions after stating his actual Theorem 7.3.}]
\label{theorem_classification_of_homogeneous_functors}
The essential image of this functor is $\homog$, and the functor is an equivalence
$$\spc^{\BO(n)} \to \homog$$
\end{theorem}

\begin{defprop}
Thus, the composition of 
$$L_n \colon \fun(\cj,\cs) \to \homog$$
with an inverse to this equivalence defines a functor which we call
$$\Theta^n \colon \fun(\cj,\cs) \to \spc^{\BO(n)}$$
For a functor $F$, we call $\Theta^n F$ its $n$-th derivative spectrum.
\end{defprop}

\begin{remark}
This \enquote{a posteriori definition} is quite far from the original definition of $\Theta^n$ in \cite{weiss-oc}, because it only makes sense once the properties of orthogonal calculus such as  \cref{theorem_classification_of_homogeneous_functors} have been established. \\
In the original article, the derivative spectrum $\Theta^nF$ of a functor $F$ is constructed long before homogeneous functors are classified. \\
However, as we do not use the original construction much\footnote{
It will appear once in the proof of \cref{derivative_of_fixed_points_of_homogeneous_functor} and is also mentioned in \cref{subsection_the_original_version_(Weiss)}.
}, we do not describe it in detail.
\end{remark}

\section{Some Topos-Theoretical Remarks}
\sectionmark{Topos-Theoretical Remarks}
\label{topos_theoretical_remarks}

This short section contains some remarks about the categories $\poly^{\leq n}(\cj,\cs)$ from a topos-theoretical viewpoint. It is not essential to the rest of the thesis and could be safely skipped.

A trivial observation is that since
$$\fun(\cj,\cs)=\fun((\cj^\text{op})^\text{op},\cs)=\text{PSh}(\cj^\text{op})$$
is a presheaf category (valued in spaces), it is an example of an $\infty$-topos.

Recall the definition:

\begin{definition}[Definition 6.1.0.4\ of \cite{lurie-htt}]
An $\infty$-category $\mathcal{X}$ is called an $\infty$-topos if there exists a small $\infty$-category $\mathcal{C}$ and an accessible left exact localisation functor $\mathcal{P}(\mathcal{C}) \to \mathcal{X}$.
\end{definition}

In $1$-category theory, this definition of a topos is equivalent to the following other definition:

\begin{definition}[{\cite[Prop.~6.1.0.1]{lurie-htt}}]
In classical category theory, a category $\cC$ is called a topos if the following equivalent conditions are satisfied
\begin{itemize}
\item $\cC$ is equivalent to the category of sheaves of sets on a Grothendieck site.
\item $\cC$ is equivalent to a left exact localisation of the category of presheaves of sets on some small category $\cC_0$.
\item Giraud's axioms are satisfied.
\end{itemize}
\end{definition}

\begin{proposition}
For every $n \in \mathbb{N}$, $\poly^{\leq n}(\cj,\cs)$ is an $\infty$-topos and 
$$T_n \colon \fun(\cj,\cs) \to \poly^{\leq n}$$ 
is a left-exact localisation.
\end{proposition}

\begin{proof}
This follows directly from \cref{infty_oc_taylor_tower} and the definition.
\end{proof}

In contrast to the $1$-categorical case, not every $\infty$-topos arises from a Grothendieck topology. Recall that there are two classes of localisations of $\infty$-topoi:

\begin{definition}[\cite{lurie-htt}, 6.5.2.17]
Let $X$ be an $\infty$-topos and let $Y \subseteq X$ be an accessible left exact localisation of $X$. $Y$ is a  \textbf{cotopological localisation} of $X$ if the localisation $X \to Y$ does not invert any monomorphisms which were not already equivalences in $X$.
\end{definition}

\begin{definition}[\cite{lurie-htt}, 6.2.1.4]
Let $X$ be an $\infty$-topos and let $Y \subseteq X$ be an accessible left exact localisation of $X$. $Y$ is a  \textbf{topological localisation} of $X$ if any morphism that is inverted by the localisation $X \to Y$ is a monomorphism in $X$. 
\end{definition}

Any (accessible) left exact localisation of $\infty$-topoi can be factored into a topological localisation followed by a cotopological localisation (\cite{lurie-htt}, 6.5.2.19). Topological localisations of $\text{PSh}(\cC)$ are precisely given by the categories of sheaves for a Grothendieck topology on $\cC$ (\cite{lurie-htt}, 6.2.2.17).

Thus, it makes sense to ask the following question.

\begin{question}
Is there a Grothendieck topology on $\cj^{\text{op}}$ such that 
$$\poly^{\leq n}(\cj,\cs) \into \text{PSh}(\cj^\text{op})$$
 is the corresponding category of sheaves, i.e.\ is it a topological localisation?
\end{question}

\begin{answer}
No. \\
It follows from work of Anel, Biedermann, Finster and Joyal 
(\cite{abfjII}[Example 4.1.19] and \cite{abfjIII}[Thm. 4.3.1(2), Cor. 4.3.4])
that the localisations occuring in each stage of the Taylor tower have the same topological components under this factorisation, hence the difference lies in the cotopological components. This topological component is the smallest left-exact localisation which makes all representable maps $\infty$-connected.
\end{answer}

\begin{remark}
This is in contrast to the situation in Manifold calculus, where $n$-polynomial functors are given as sheaves for the Grothendieck topology $\cj_n$, where a cover $\{U_i \to U\}_{i \in I}$ consists of a family of maps with the property that for any subset $X \subseteq U$ of cardinality $\leq n$, there exists an $i \in I$ such that $X \subseteq \text{im}(U_i)$. (This is the main result of \cite{debrito-weiss}.) 

In particular, Manifold Calculus is thus not a Generalised Goodwillie calculus in the sense of Anel, Biedermann, Finster and Joyal.
\end{remark}

\chapter{Orthogonal Calculus and Monoidal Structures} 
\chaptermark{Monoidal Structures}
\label{chapter_monoidal_oc_theorem}

The goal of this chapter is to prove \cref{orthogonal_monoidality_theorem}, i.e.\ that the $n$-polynomial approximation functors
$$T_n \colon \fun(\cj,\cs) \to \fun(\cj,\cs)$$
are lax symmetric monoidal.

Before, we will need some preliminaries:

\begin{itemize}
\item In \cref{subsection-on-symmetric-monoidal-infinity-categories} we briefly recall some fundamentals about symmetric monoidal infinity categories and functors between them, and explain the symmetric monoidal structure on $\cj$.
\item In \cref{subsection-on-day-convolution} we explain the Day convolution symmetric monoidal structure on $\fun(\cj,\cs)$, and in \cref{subsection_on_internal_hom} we prove that in this seting, it admits an internal hom functor with particularly nice properties.
\item In \cref{subsection-on-general-monoidal-localisations} we state and explain some abstract nonsense about the question when a localisation of a monoidal category is itself monoidal.
\end{itemize}

Then we can prove the theorem in \cref{section_on_monoidal_localisation_theorem}.

Afterwards, in \cref{monoidal_oc_spectra}, we will show that the result carries over homogeneous layers, and prove that this induces \enquote{twisted multiplication maps}
$$\Theta^nF \otimes \Theta^nF \to D_{O(n)} \otimes \Theta^nF$$
on the derivative spectra of a monoidal functor $F \colon \cj \to \cs$.

\section{Setup / Preliminaries}

\subsection{Symmetric Monoidal $\infty$-Categories}
In the following, we use the language and techniques of symmetric monoidal $\infty$-categories in the sense of Lurie  (\cite[Chapter~2]{lurie-ha}).

We recall the most important definition:

\begin{definition}[\cite{lurie-ha}, 2.0.0.7]
A symmetric monoidal $\infty$-categeory is a coCartesian fibration of simplicial sets $p \colon \cC^\otimes \to N(\text{Fin}_*)$ with the property:
\begin{center}
For each $n \geq 0$, the maps $\{p_i \colon \langle n \rangle \to \langle 1 \rangle\}_{1 \leq i \leq n}$ induce funtors $\rho^i_! \colon \cC^\otimes_{\langle n \rangle} \to \cC^\otimes_{\langle 1 \rangle}$ which determine an equivalence $\cC^\otimes_{\langle n \rangle} \simeq \left( \cC^\otimes_{\langle 1 \rangle} \right)^n$.
\end{center}
\end{definition}

Any symmetric monoidal $\infty$-category is an example of an $\infty$-operad.

In this generality, a lax symmetric monoidal functor is an $\infty$-operad map (see also {\cite[Remark~2.1.3.6]{lurie-ha}}).

\begin{definition}[\cite{lurie-ha}, 2.1.2.7]
Let $\mathcal{O}^\otimes$ and $\mathcal{O}'^\otimes$ be $\infty$-operads. An $\infty$-operad map from $\mathcal{O}^\otimes$ to $\mathcal{O}'^\otimes$ is a map of simplicial sets $\mathcal{O}^\otimes \to \mathcal{O}'^\otimes$ which commutes with the structure maps to $N(\text{Fin}_*)$ and which carries inert morphisms to inert morphisms. (A morphism $\langle m \rangle \to \langle n \rangle$ is called inert if it identifies a subset of $\langle m \rangle$ with the basepoint of $\langle n \rangle$ and is otherwise bijective. A morphism in an $\infty$-operad is inert if it maps to an inert morphism of $N(\text{Fin}_*)$.)
\end{definition}

\label{subsection-on-symmetric-monoidal-infinity-categories}

\paragraph{The Symmetric Monoidal Structure on $\cj$}
\label{subsection-on-direct-sum-on-j}

We want to equip $\cj$ with the structure of a symmetric monoidal $\infty$-category corresponding to the construction
$$\left( V,\langle -,-\rangle_V), (W,\langle -,-\rangle_W)\right) \mapsto (V \oplus W,\langle -,-\rangle_V \oplus \langle -,- \rangle_W)$$
We simply write $V \oplus W$ for this, noting however that this is not actually a coproduct in $\cj$.

\begin{warning}
\label{direct-sum-on-cj-is-not-coproduct}
$\cj$ does not admit any coproducts of inner product spaces of positive dimension. 
\end{warning}
\begin{proof} Let $V$ be an inner product space of dimension $n > 0$ and assume $X$ to be a coproduct of $V$ and $V$. $X$ admits a codiagonal map to $V$, thus $\dim X \leq n$ because all maps are embeddings. On the other hand, consider the two standard embeddings $f,g \colon V \to V \oplus V$. Clearly, $\dim(V \oplus V) = 2n$. By the universal property, there is an embedding $f \oplus g \colon X \to V \oplus V$ extending $f$ and $g$, thus $\dim X = 2n$. \\
The case of a coproduct of different inner product spaces works analogously after choosing an embedding of one into the other.
\end{proof}

{\sloppy
\begin{propdef}
There is a symmetric monoidal $\infty$-category $(\cj^\oplus)$ with underlying $\infty$-category $\cj$  extending the construction 
$$(V,\langle -,-\rangle_V) \times (W,\langle -,-\rangle_W) \mapsto (V \oplus W,\langle -,-\rangle_V \oplus \langle -,- \rangle_W)$$
\end{propdef}
}

\begin{proof}
We will explicitly describe a cocartesian fibration $\cj^\oplus \to \nfin$ such that $\cj^\oplus_n \cong (\cj^\oplus_1)^n$, i.e.\ verify \cite[Definition 2.0.0.7.]{lurie-ha}.

An object of $\cj^\oplus$ consists of a finite pointed set $X$ together with objects $(X_i)_{i \in X} \in \cj$. A morphism of $\cj^\oplus$ consists of a map of pointed sets $f \colon X \to Y$ together with maps $\tilde{f}_k \colon \bigoplus_{i \in f^{-1}(k)} X_i \to Y_k$ for $k \in Y \backslash \{*\}$. The functor to $\nfin$ is the one which remembers only the finite pointed set.

The property $$\cj^\oplus_n \cong (\cj^\oplus_1)^n$$ follows directly from the definition.

Next, $\cj^\oplus \to \nfin$ has to be an inner fibration. The required lifting property follows, after unravelling the definitions, directly from the fact that $\cj$ is an $\infty$-category (see also \cite[Proof of 2.4.3.3.]{lurie-ha}).

Finally, $\cj^\oplus \to \nfin$ should be a cocartesian fibration. Thus, let $f \colon X \to Y$ be a morphism in $\nfin$ and $\hat{X} \in \cj^\oplus$ a lift of $X$. We need to find a cocartesian lift $\hat{f} \colon \hat{X} \to \hat{Y}$. We let $(\hat{Y})_i = \bigoplus_{f^{-1}(i)}\hat{X}_i$, and $\hat{f}$ the canonical map (consisting of identity maps). $\hat{f}$ is cocartesian, as can be verified using the definition of $\cj^\oplus$ and \cite[Prop.\ 2.4.4.3.]{lurie-htt}.
\end{proof}

\subsection{Day Convolution and the Internal Hom Functor}
\subsubsection{The Day Convolution Monoidal Structure}
\label{subsection-on-day-convolution}
We will now put a monoidal structure on $\ce$, called the \textit{Day convolution monoidal structure}, which has the important property that monoids with respect to it are precisely lax monoidal functors. 

Day convolution goes back to Brian Day in the 1-categorical setting (\cite{day-article, day-thesis}) and has been made available in the $\infty$-categorical world first by Saul Glasman (\cite{glasman}) and in greater generality by Jacob Lurie (\cite[Section~2.2.6]{lurie-ha}).

We will first recall some generalities. The reader who is already familiar with Day convolution may want to skip this section.

\paragraph{The case of symmetric monoidal 1-categories} $ $ \\
For sufficiently nice\footnote{
i.e.\ $\cC$ small, $\cD$ admits all small colimits}
 symmetric monoidal categories $\cC$ and $\cD$ with tensor product functors
\begin{center}
$ \otimes_\cC \colon \cC \times \cC \to \cC$, $\otimes_\cD \colon \cD \times \cD \to \cD$
\end{center}
and arbitrary functors $$F,G \colon \cC \to \cD$$ their \textit{Day convolution product} $$F \conv G \in \fun (\cC,\cD)$$ is defined to be the left Kan extension of $\otimes_\cD \circ (F \times G)$ along $\otimes_\cC$ as shown in the diagram

\begin{center}
\begin{tikzcd}[column sep = 5em, row sep = 5em]
\cC \times \cC \arrow[rr, "F \times G"] \arrow[dd, "\otimes_\cC"] \arrow[rrdd, "\otimes_\cD \circ (F \times G)", ""'{name=U}] &  & \cD \times \cD \arrow[dd, "\otimes_\cD"] \\
                                                                                                                 &  &                                          \\
\cC \arrow[rr, "F \conv G := \Lan_{(\otimes_\cD \circ (F \times G))}\otimes_\cC"]                          
\arrow[Rightarrow, from=U, shorten > = 2em, shorten < = 1em] &  & \cD                                     
\end{tikzcd}
\end{center}

where the upper triangle commutes and the lower triangles commutes up to a universal natural transformation 
$$\otimes_\cD \circ (F \times G) \Rightarrow (\Lan_{(\otimes_\cD \circ (F \times G))} \otimes_\cC) \circ \otimes_\cC$$

\begin{defprop}[Day \cite{day-thesis}, Lurie {\cite[Example 2.2.6.17]{lurie-ha}}]$ $ \\
The association
$$ - \conv - \colon \fun(\cC,\cD) \times \fun(\cC,\cD) \to \fun(\cC,\cD)$$
determines a symmetric monoidal structure on $\fun(\cC,\cD)$ 
such that monoid objects of $(\fun(\cC,\cD),\conv)$ are precisely the lax symmetric monoidal functors.
\end{defprop}

Using the pointwise formula for Kan extensions, there is a more explicit description.

\begin{proposition}[Day, Lurie {\cite[Remark 2.2.6.15]{lurie-ha}}]$ $ \\
\label{pointwise-day-convolution}
Let $F,G \colon \cC \to \cD$, then $F \conv G \colon \cC \to \cD$ is given on objects by

$$(F \conv G)(C) = \colim_{C_0 \otimes_\cC C_1 \to C} F(C_0) \otimes_\cD G(C_1)$$
\end{proposition}

\paragraph{The case of symmetric monoidal $\infty$-categories} $ $ \\
Analogous statements hold in the world of symmetric monoidal (or more generally $\mathcal{O}$-monoidal) $\infty$-categories. We will use the constructions of Lurie. References are in particular \cite[Construction 2.2.6.7, Example 2.2.6.9, Example 2.2.6.17]{lurie-ha}. We reproduce the latter here for convenience

\begin{proposition}[Lurie, {\cite[Example 2.2.6.9]{lurie-ha}}]
\label{infty-day-convolution}
Let $\cC^\otimes$ be a symmetric monoidal $\infty$-category and let $\cD^\otimes$ be an arbitrary $\infty$-operad. Applying Construction 2.2.6.7 of \cite{lurie-ha} in the case where $\mathcal{O}^\otimes = \text{Comm}^\otimes$ is the commutative $\infty$-operad, we obtain an $\infty$-operad $\fun(\cC,\cD)^\otimes$ with the following features:
\begin{itemize}
\item The underlying $\infty$-category of $\fun(\cC,\cD)^\otimes$ is canonically equivalent to the $\infty$-category $\fun(\cC,\cD)$ of functors from $\cC$ to $\cD$.
\item The $\infty$-category $\calg(\fun(\cC,\cD))$ of commutative algebra objects of $\fun(\cC,\cD)$ is equivalent to the $\infty$-category $\text{Alg}_\cC(\cD)$ of lax symmetric monoidal functors from $\cC$ to $\cD$.
\end{itemize}
\end{proposition}

\paragraph{Day convolution on $\ce$ and some properties} $ $ \\
As a special case of the previous discussion we obtain
\begin{proposition}
$\fun(\cj,\cs)$ and 
$\fun(\cj,\csp)$ are symmetric monoidal $\infty$-categories with respect to Day convolution.
\end{proposition}

An example of functors whose convolution product is easy to understand are representable functors.

\begin{example}
\label{convolution-of-corepresentables}
$\mor(V,-) \conv \mor(W,-) \simeq \mor(V \oplus W,-)$
\end{example}

\begin{proof}
By the pointwise formula (\ref{pointwise-day-convolution}), we see that
$$
(\mor(V,-) \conv \mor(W,-))(X) \simeq \colim_{A \oplus B \to X} \mor(V,A) \times \mor(W,B)
$$
The terms on the right-hand side admit maps
$$\mor(V,A) \times \mor(W,B) \to \mor(V \oplus W,A \oplus B) \to \mor(V \oplus W,X)$$
which assemble to a map 
$$ \colim_{A \oplus B \to X} \mor(V,A) \times \mor(W,B) \to \mor(V \oplus W,X)$$
This map is an equivalence. An inverse is given by
\begin{align*}
\mor(V \oplus W,X) &\to \mor(V,\im_f(V)) \times (W,\im_f(W)) \\
&\to \colim_{A \oplus B \to X} \mor(V,A) \times \mor(W,B)\\
(f \colon V \oplus W \to X) &\mapsto (f_{|V},f_{|W})
\end{align*}

\end{proof}

Another important property of Day convolution in our setting is that it commutes with colimits in each variable.
\begin{lemma}
\label{convolution-commutes-with-colimits}
On $\fun(\cj,\cs)$, the convolution product $\conv$ commutes with colimits in each variable separately.
\end{lemma}

\begin{proof}
Roughly, this follows from \cref{pointwise-day-convolution} by commuting colimits with colimits and using that the (monoidal) product in $\cs$ commutes with colimits as $\cs$ is closed monoidal. \\
Recall the definition
$$F \conv G \cong \lan_{\otimes_J} \otimes_\mathcal{S} \circ (F \times G)$$
Since $\lan_{\otimes_\cj}$ is left adjoint to $\otimes_\cj^* \colon [\mathcal{J},\mathcal{S}] \to [\mathcal{J}\times \mathcal{J}, \mathcal{S}]$, it commutes with colimits in $\fun(J \times J,\mathcal{S})$. To have $\conv$ commute with colimits in each variable we thus need
$$ (\colim_I F_i) \times G \simeq \colim_I (F_i \times G)$$
(and its analogue for $G_i$) in $\fun(J \times J,S)$. This holds as colimits are computed in the codomain category and since $\mathcal{S}$ is closed monoidal, hence $ - \times G$ preserves colimits. 
\end{proof}

\subsubsection{The Internal Hom Functor}
\label{subsection_on_internal_hom}

\begin{corollary}
\label{internal_hom}
$\fun(\cj,\cs)$ equipped with $\conv$ is closed monoidal, i.e.\ there is an internal hom functor 
$$\inthom \colon \fun(\cj,\cs)^{\text{op}} \times \fun(\cj,\cs) \to \fun(\cj,\cs)$$
such that for $X \in \fun(\cj,\cs)$, there is an adjunction with left adjoint
$$X \conv - \ \colon \fun(\cj,\cs) \to \fun(\cj,\cs)$$
and right adjoint
$$\inthom(X,-) \colon \fun(\cj,\cs) \to \fun(\cj,\cs)$$
\end{corollary}

\begin{proof}
For fixed $X$, the existence of an adjoint to $X \conv -$ follows from \cref{convolution-commutes-with-colimits} using the adjoint functor theorem (Corollary 5.5.2.9 of \cite{lurie-htt}).
{\sloppy That all of those adjoints assemble into a single internal hom functor $\inthom$ follows since an adjoint exists if it exists for every object (see e.g.\ \cite[Prop.~6.1.11]{cisinski-book}). }
\end{proof}

\begin{remark}
\label{internal_hom_commutes_with_products}
For every $X$, $\inthom(X,-) \colon \fun(\cj,\cs) \to \fun(\cj,\cs)$ commutes with limits, as it is a right adjoint.
\end{remark}
\begin{proposition}
\label{internal_hom_remains_poly}
For an arbitrary functor $A \in \fun(\cj,\cs)$ and an $n$-polynomial functor $P \in \poly^{\leq^n}(\cj,\cs)$, the internal hom
$$\inthom(A,P) \in \fun(\cj,\cs)$$
is $n$-polynomial.
\end{proposition}

\begin{proof}
By \cref{polynomial_of_degree_n}, we have to check that
$$\inthom(A,P)(X) \to \lim_{0 \neq U \subseteq \mathbb{R}^{n+1}} \inthom(A,P)(X \oplus U) = \lim_{0 \neq U \subseteq \mathbb{R}^{n+1}} \sh_U^* \inthom(A,P) $$
is an equivalence, where $\sh_U$ denotes the functor $\sh_U \colon V \mapsto V \oplus U$
and $\sh_U^* \colon \fun(\cj,\cs) \to \fun(\cj,\cs)$ denotes precomposition by $\sh_U$.

Assume for the moment that 
$$\sh_U^* \inthom(A,P) \simeq \inthom(A,\sh_U^* P)$$
Then, we could calculate
\begin{align*}
\lim_{0 \neq U \subseteq \mathbb{R}^{n+1}}\left(\sh_U^* \inthom(A,P) \right)
	&\simeq \lim_{0 \neq U \subseteq \mathbb{R}^{n+1}} \inthom(A,\sh_U^* P ) \\
	&\simeq_{\cref{internal_hom_commutes_with_products}} \inthom \left( A, \lim_{0 \neq U \subseteq \mathbb{R}^{n+1}} \sh_U^* P \right) \\
	&\simeq_{P \in \poly^{\leq n}(\cj,\cs)} \inthom(A,P)
\end{align*}
and thus conclude that $\inthom(A,P)$ is $n$-polynomial.

It remains to prove that
$$\sh_U^* \inthom(A,P) \simeq \inthom(A,\sh_U^* P)$$

We observe that 
$$\sh_U^* \colon \fun(\cj,\cs) \to \fun(\cj,\cs)$$ has a left adjoint 
$$(\sh_{U})_! \colon \fun(\cj,\cs) \to \fun(\cj,\cs)$$
 given by left Kan extension along $\sh_U$, which exists as $\cs$ is cocomplete. Using the pointwise formula for Kan extensions, it is given by
$$(\sh_U)_! F (V) \simeq \colim \left( \sh_U/V \to \cj \to \cs \right) \simeq \colim_{X \oplus U \to V} F(X) $$

Using the adjunction $(\sh_U)_! \dashv \sh_U^*$ we can compute
\begin{align*}
\Nat(A,\sh_U^* \inthom(B,C)) 
	&\simeq \Nat((\sh_U)_! A,\inthom(B,C)) \\
	&\simeq \Nat(((\sh_U)_! A) \conv B,C) \\
	&\simeq \Nat(((\sh_U)_! (A \conv B),C) \\
	&\simeq \Nat(A \conv B,\sh_U^*C) \\
	&\simeq \Nat(A,\inthom(B, \sh_U^*C)) 
\end{align*}

where we used that 
$$(\sh_U)_! A \conv B \simeq (\sh_U)_! (A \conv B)$$
as Day convolution commutes with colimits and the left Kan extension $(\sh_U)_!$ is given pointwise by colimits.

Thus, by the Yoneda lemma
$$\sh_U^* \inthom(B,C) \simeq \inthom(B, \sh_U^*C)$$

\end{proof}

\paragraph{A Formula for the Internal Hom Functor} $ $ \\

In \cref{internal_hom} and \cref{internal_hom_remains_poly}, we only proved and used the existence of the internal hom functor $\inthom \colon \fun(\cj,\cs)^{op} \times \fun(\cj,\cs) \to \fun(\cj,\cs)$, but did not provide a formula. Luckily, $\fun(\cj,\cs)$ is a (co)presheaf category, where this is relatively easy.

\begin{proposition}
\label{pointwise_formula_for_internal_hom}
For $A,B \in \fun(\cj,\cs)$ and $X \in \cj$
$$\inthom(A,B)(X) \simeq \Nat(A,\sh_X^*B) $$
\end{proposition}

\begin{proof}
Denote by 
\begin{align*}
\yo \colon \cj^{op} &\to \fun(\cj,\cs) \\
V &\mapsto \Map(V,-)
\end{align*}

the contravariant Yoneda embedding.

Observe first that there is an equivalence
$$\inthom(A,B)(X) \simeq \Nat(\yo(X) \conv A,B)$$
which is a simple consequence of the Yoneda lemma, as
\begin{align*}
\inthom(G,H)(X) &\simeq
	\Map_{PSh(\cj^{op})}(\yo(X),\inthom(G,H)) \\
	&\simeq \Map_{PSh(\cj^{op})}(\yo(X) \conv G,H) \\
	&\simeq \Nat(\yo(X) \conv G,H)
\end{align*}

Next, recall from \cref{convolution-of-corepresentables} that
$$\yo(X) \conv \yo(Y) \simeq \yo(X \oplus Y)$$
 for $X,Y \in \cj$. Thus, again by Yoneda, 
$$\inthom(\yo (Y), A)(X) \simeq A(Y \oplus X)$$

Finally, the claim follows by calculating
\begin{align*}
\inthom(A,B)(X) 
	&\simeq \Map(\yo(X),\inthom(A,B)) \\
	&\simeq \Map(\yo(X) \conv A,B) \\
	&\simeq \Map(A,\inthom(\yo(X),B)) \\
	&\simeq \Nat(A,\sh_X^*B)
\end{align*}
\end{proof}

\begin{remark}
One can also use this formula to give another proof of  \cref{internal_hom_remains_poly}. Recall that we had to check that
$$\sh_U^*\inthom(A,P) \simeq \inthom(A,\sh_U^*P)$$
Indeed, using \cref{pointwise_formula_for_internal_hom}, we see that
$$\sh_U^*\inthom(A,P)(X) \simeq \Nat(A,\sh_{X \oplus U}^*P)$$
and that
$$\inthom(A,\sh_U^*P)(X) \simeq \Nat(A,\sh_X^*\sh_U^*P)$$
The claim follows as $\sh_U \sh_X \simeq \sh_{X \oplus U}$.
\end{remark}

\subsection{Generalities on Monoidal Localisations}
\label{subsection-on-general-monoidal-localisations}

To exhibit the localisation $T_n$ as monoidal, we will use the following abstract result from category theory, which is essentially due to Brian Day (\cite{day-note-on-monoidal-localisations}, Corollary 1.4) in the 1-categorical setting and has been carried over to $\infty$-land by Lurie.

\begin{definition}[\cite{day-note-on-monoidal-localisations}, Definition 1.3]
Let $Z$ be a class of morphisms in a symmetric monoidal category $\mathcal{A}$. Then $Z$ is called monoidal if 
$$Z = \{s \in Z \mid \id_A \otimes s \in Z \text{ for all } A \in \mathcal{A} \}$$
or equivalently if the class $Z$ is closed under the tensor product.
\end{definition}

\begin{proposition}[\cite{day-note-on-monoidal-localisations}, Corollary 1.4]
If $Z$ is monoidal in $\mathcal{A}$, then there is a monoidal structure on $\mathcal{A}[Z^{-1}]$ such that the projection functor $\mathcal{A} \to \mathcal{A}[Z^{-1}]$ is a monoidal functor.
\end{proposition}

\begin{proof}[Proof sketch]
Under the assumption on $Z$, it is straightforward to check that the upper map in the diagram
\[\begin{tikzcd}
	{\mathcal{A} \times \mathcal{A}} && {\mathcal{A}} \\
	{\mathcal{A}[Z^{-1}] \times \mathcal{A}[Z^{-1}]} && {\mathcal{A}[Z^{-1}]}
	\arrow["{\otimes_\mathcal{A}}", from=1-1, to=1-3]
	\arrow[from=1-1, to=2-1]
	\arrow["{\otimes_{\mathcal{A}[Z^{-1}]}}", dashed, from=2-1, to=2-3]
	\arrow[from=1-3, to=2-3]
\end{tikzcd}\]
sends morphisms with components in $Z$ to equivalences, hence the universal property induces the required factorisation.
\end{proof}

In $\infty$-land, the analogous statement follows from \cite[Prop.~2.2.1.9]{lurie-ha}. In the literature, it appears as \cite[Lemma~3.4]{gepner-groth-nikolaus}. The (equivalent) version which we find most convenient appears in Eigil Rischel's bachelor thesis \cite[Section~8]{rischel}.

\begin{definition}[{\cite[Def.~2.2.1.6]{lurie-ha}}, {\cite[Def.~3.3]{gepner-groth-nikolaus}},{\cite[Def.~8.1, Rmk.~8.2]{rischel}}]
\label{compatible_localisation}
Let $\cC^\otimes$ be a symmetric monoidal $\infty$-category and $L \colon \cC \to \cC$ be a localisation of the underlying $\infty$-category. Then we say that $L$ is \textit{compatible with the symmetric monoidal structure} if the following equivalent conditions are satisfied:
\begin{itemize}
\item The tensor product of local equivalences\footnote{
Recall that morphism $f \colon C_1 \to C_2$ in $\cC$ is a local equivalence if $L(f)$ is an equivalence.
} is a local equivalence.
\item Whenever $f \colon X \to Y$ is a local equivalence, and $Z$ is any object of $\cC$, the map 
$$f \otimes \id_Z \colon X \otimes Z \to Y \otimes Z$$
is also a local equivalence.
\end{itemize}
\end{definition}

\begin{proposition}[{\cite[Prop.~2.2.1.9]{lurie-ha}}, {\cite[Lemma~3.4]{gepner-groth-nikolaus}}, {\cite[Prop.~8.3]{rischel}}]
\label{compatible_localisation_consequence}
Let $p \colon \cC^\otimes \to \nfin$ be a symmetric monoidal $\infty$-category and $L \colon \cC \to \cC$ a localisation of $\cC$ which is compatible with the symmetric monoidal structure. Let $(L\cC)^\otimes$ be the subcategory of $\cC^\otimes$ spanned by objects of the form $(LC_1, \ldots, LC_n)$. Then
\begin{itemize}
\item The restriction $p \colon (L\cC)^\otimes \to \nfin$ exhibits $(L\cC)^\otimes$ as a symmetric monoidal $\infty$-category with underlying $\infty$-category $L \cC$.
\item The inclusion $(L\cC)^\otimes \to \cC^\otimes$ admits a left adjoint $L^\otimes$ such that $(L^\otimes)_{|\cC} \simeq L$ and the unit $\eta$ of the adjuncton can be chosen such that $p(\eta) = 1$.
\item $L^\otimes$ is a (strict) symmetric monoidal functor.
\item The inclusion $(L\cC)^\otimes \subset \cC^\otimes$ is a lax symmetric monoidal functor.
\end{itemize}
\end{proposition}

\begin{proof}[Proof sketch]
The key step is to construct the left adjoint $L^\otimes \colon \cC^\otimes \to (L\cC)^\otimes$. This uses that $\infty$-categorical adjoints exist once each object admits an adjoint object\footnote{
This is where the coherences required for a functor of $\infty$-categories are taken care of. See e.g.\ Proposition 6.1.11 of \cite{cisinski-book}.
}. For an object $(C_1, \ldots, C_k) \in \cC^\otimes$,
$$(\eta_{C_1}, \ldots, \eta_{C_k}) \colon (C_1, \ldots, C_k) \to (LC_1, \ldots, LC_k)$$
(where $\eta$ is the unit of the original adjunction)
is a localisation with respect to $(L\cC)^\otimes$, which can be checked by considering an arbitrary object $(Y_1, \ldots, Y_l)$ and checking that the induced map
$$ \Map_{\cC^\otimes}((LC_1, \ldots, LC_k),(Y_1, \ldots, Y_l)) \to \Map_{\cC^\otimes}((C_1, \ldots, C_k),(Y_1, \ldots, Y_l))$$
is an equivalence, which can be checked componentwise, where it follows from the assumption. \\
The remainder of the proof consists of some considerations of coCartesian morphisms and fibrations, using \cite[Lemma~2.2.1.11]{lurie-ha}.
\end{proof}

\begin{corollary}[{\cite[Lemma~3.6, Rmk.~3.7]{gepner-groth-nikolaus}}, , {\cite[Cor.~8.6, Cor.~8.8]{rischel}}]
Let $\cC^\otimes$ be a symmetric monoidal $\infty$-category equipped with a symmetric monoidal localisation $L \colon \cC \to L\cC$ and let $R \colon L\cC \to \cC$ be the right adjoint. Then there is an induced localisation $L' \colon \calg(\cC) \to \calg(L\cC)$ such that the diagram 
\[\begin{tikzcd}
	{\calg(\cC)} & {\calg(L\cC)} \\
	\cC & L\cC
	\arrow["L", from=2-1, to=2-2]
	\arrow["{L'}", from=1-1, to=1-2]
	\arrow[from=1-1, to=2-1]
	\arrow[from=1-2, to=2-2]
\end{tikzcd}\]
commutes.
Moreover, given $A \in \calg(\cC)$, there exists a unique commutative algebra structure on $RLA$ such that the unit map $A \to RLA$ extends to a morphism of commutative algebras.
\end{corollary}

\begin{proof}
An $E_\infty$-algebra in $\cC$ is given by a section of $p \colon \cC^\otimes \to \nfin$. Composing such a section with $\cC^\otimes \to (L\cC)^\otimes$ gives a section of $(L\cC)^\otimes \to \nfin$, i.e.\ an $E_\infty$-algebra in $L\cC$. \\
The second part follows from the fact that a map of algebras is an equivalence if the underlying map is an equivalence and the universal property of the unit.
\end{proof}

Now that we have stated the rigorous details, let's step back for a moment to reflect on what we just said.

\begin{remark}[Informal summary]
If a localisation $L \colon \cC \to L\cC$ and a monoidal structure $\otimes_\cC$ on $\cC$ are compatible in the sense that local equivalences are closed under the tensor product, then:
\begin{itemize}
\item There is a symmetric monoidal structure on $L\cC$ given by
$$LC_1 \otimes_{L\cC} LC_2 := L(LC_1 \otimes_\cC LC_2)$$
\item The lax monoidal structure on $\iota \colon L\cC \to \cC$ is given by the unit of the adjunction, i.e.\ by the maps
$$\eta_{LC_1 \otimes_\cC LC_2} \colon LC_1 \otimes_\cC LC_2 \to L(LC_1 \otimes_\cC LC_2) \simeq LC_1 \otimes_{L\cC} LC_2$$
\item $L \colon \cC \to L\cC$ is (strict) symmetric monoidal, as under the hypothesis, there are equivalences
$$L(C_1 \otimes_\cC C_2) \to L(LC_1 \otimes_\cC LC_2) \simeq LC_1 \otimes_{L\cC} LC_2 $$
\end{itemize}
\end{remark}

\section{A Monoidal Localisation Theorem for Taylor Approximations in Orthogonal Calculus}
\sectionmark{Monoidal Localisation Theorem}
\label{section_on_monoidal_localisation_theorem}

Recall that Orthogonal Calculus provides left-exact localisations

\[\begin{tikzcd}
	{\fun(\cj,\cs)} & {\poly^{\leq n}(\cj,\cs)}
	\arrow[""{name=0, anchor=center, inner sep=0}, "{T_n}", shift left=2, from=1-1, to=1-2]
	\arrow[""{name=1, anchor=center, inner sep=0}, "\iota", shift left=2, hook', from=1-2, to=1-1]
	\arrow["\dashv"{anchor=center, rotate=-90}, draw=none, from=0, to=1]
\end{tikzcd}\]

Our original goal was to show that whenever $F \colon \cj \to \cs$ is a lax monoidal functor (for example $V \mapsto \text{O}(V)$ or $V \mapsto \text{TOP}(V)$), so is $T_nF$. In light of \ref{infty-day-convolution}, it is enough to show that $T_n \colon \fun(\cj,\cs) \to \fun(\cj,\cs)$ is lax monoidal with respect to Day convolution, as a lax monoidal functor sends monoids to monoids. Hence, our theorem now takes the following form

\begin{theorem}
\label{orthogonal_monoidality_theorem}
The $n$-polynomial approximation functors 
$$T_n \colon \fun(\cj,\cs) \to \poly^{\leq n}(\fun(\cj,\cs))$$
are (strict) symmetric monoidal with respect to the Day convolution monoidal structure on $\fun(\cj,\cs)$ and an induced monoidal structure on $\poly^{\leq n}(\fun(\cj,\cs))$.
The inclusion
$$\iota \colon \poly^{\leq n}(\fun(\cj,\cs)) \to \fun(\cj,\cs)$$
 is lax symmetric monoidal. The composition\footnote{
which is often just called $T_n$ as well
 }
$$\iota \circ T_n \colon \fun(\cj,\cs) \to \fun(\cj,\cs)$$
is lax symmetric monoidal.
\end{theorem}

\begin{corollary}
\label{approx_of_monoidal_functors}
The $n$-polynomial approximation $T_nF \colon \cj \to \cs$ of a lax symmetric monoidal functor $F \colon \cj \to \cs$ is lax symmetric monoidal.
\end{corollary}

Using what we explained in the previous section, it is sufficient to show that the localisation $T_n$ is compatible with the Day convolution monoidal structure on $\fun(\cj,\cs)$ in the sense of \cref{compatible_localisation}.

\begin{proposition}
\label{orthogonal_monoidal_compatibility}
$T_n$ is compatible with the Day convolution monoidal structure on $\fun(\cj,\cs)$ in the sense of \cref{compatible_localisation}.
\end{proposition}

The proof of \cref{orthogonal_monoidal_compatibility} will be given in the next section. Assuming it for now, the main theorem follows easily:

\begin{proof}[Proof of \cref{orthogonal_monoidality_theorem} assuming compatibility of $T_n$ and Day convolution]
By \cref{orthogonal_monoidal_compatibility}, the localisation
\[\begin{tikzcd}
	{\fun(\cj,\cs)} & {\poly^{\leq n}(\cj,\cs)}
	\arrow[""{name=0, anchor=center, inner sep=0}, "{T_n}", shift left=2, from=1-1, to=1-2]
	\arrow[""{name=1, anchor=center, inner sep=0}, "\iota", shift left=2, hook', from=1-2, to=1-1]
	\arrow["\dashv"{anchor=center, rotate=-90}, draw=none, from=0, to=1]
\end{tikzcd}\]
satisfies the conditions of \ref{compatible_localisation_consequence}. As a consequence, $\text{Poly}^{\leq n}(\cj,\cs)$ inherits a symmetric monoidal structure such that $T_n \colon \fun(\cj,\cs) \to \text{Poly}^{\leq n}(\cj,\cs)$ is symmetric monoidal, the inclusion $\iota \colon \text{Poly}^{\leq n}(\cj,\cs) \to \fun(\cj,\cs)$ is lax symmetric monoidal, $\iota \circ T_n \colon \fun(\cj,\cs) \to \fun(\cj,\cs)$ is lax symmetric monoidal. 
\end{proof}

\begin{proof}[Proof of \cref{approx_of_monoidal_functors}]
There is an equivalence between lax symmetric monoidal functors $\cj \to \cs$ and monoids in the Day convolution symmetric monoidal category (\cref{infty-day-convolution}). By \cref{orthogonal_monoidality_theorem}, $\iota \circ T_n$ is lax symmetric monoidal, hence maps monoids to monoids. The diagram looks as follows.
\[\begin{tikzcd}
	{\calg(\fun(\cj,\cs))} & {\calg(\text{Poly}^{\leq n}(\cj,\cs))} & {\calg(\fun(\cj,\cs))} \\
	{\fun^\text{lax monoidal}(\cj,\cs)} && {\fun^\text{lax monoidal}(\cj,\cs)}
	\arrow["\simeq"{description}, draw=none, from=1-1, to=2-1]
	\arrow["\simeq"{description}, draw=none, from=1-3, to=2-3]
	\arrow["{\iota \circ T_n}", from=2-1, to=2-3]
	\arrow["{T_n}", from=1-1, to=1-2]
	\arrow["\iota", from=1-2, to=1-3]
\end{tikzcd}\]
\end{proof}

\begin{proof}[Proof of \cref{orthogonal_monoidal_compatibility} (Compatibility)] $ $ \\
\textbf{Step 1: Reduction to $\eta_A \colon A \to T_nA$} \\
First, we claim that it is sufficient to show that for any $A, C \in \fun(\cj,\cs)$, the transformation
$$\eta_A \conv \id_C \colon A \conv C \to T_nA \conv C$$ 
is an $n$-polynomial equivalence\footnote{One might also call such a map a $T_n$-local equivalence.}, i.e.\ $T_n(\eta_A \conv \id_C)$ is an equivalence. \\
To see that, let $f \colon A \to B$ be an arbitrary $n$-polynomial equivalence. Consider the diagram
\[\begin{tikzcd}
	{A \conv C} && {B \conv C} \\
	{T_nA \conv C} && {T_nB \conv C}
	\arrow["{f \conv \id_C}", from=1-1, to=1-3]
	\arrow["{\eta_A \conv \id_C}"', from=1-1, to=2-1]
	\arrow["{\eta_B \conv \id_C}", from=1-3, to=2-3]
	\arrow["{T_n(f) \conv \id_C}", from=2-1, to=2-3]
\end{tikzcd}\]
The lower horizontal morphism is an equivalence to begin with, hence the upper morphism is sent to an equivalence by $T_n$ if the same is true for the left and the right one. \\
\textbf{Step 2: Yoneda} \\
Let again $A, C \in \fun(\cj,\cs)$. By Yoneda in $\poly^{\leq n}(\cj,\cs)$, it suffices to show that for any $P \in \poly^{\leq n}(\cj,\cs)$, the upper map in
\[\begin{tikzcd}[column sep = huge]
	{\Map(T_nA \conv C, P)} & {\Map(A \conv C, P)} \\
	{\Map(T_n(T_nA \conv C),P)} & {\Map(T_n(A \conv C),P)}
	\arrow["{(\eta_A \conv \id_C)^*}", from=1-1, to=1-2]
	\arrow["{(T_n(\eta_A \conv \id_C))^*}", from=2-1, to=2-2]
	\arrow["\simeq", from=2-2, to=1-2, sloped]
	\arrow["\simeq", from=2-1, to=1-1, sloped]
\end{tikzcd}\]
is an equivalence. \\

\textbf{Step 3: }$(\eta_A \conv \id_C)^*$ \textbf{is an equivalence:} \\

What remains to be done is the following simple calculation using the internal hom adjunction from \cref{internal_hom}, \cref{internal_hom_remains_poly} and the adjunction between $T_n$ and the inclusion of polynomial functors.
\begin{align*}
\Nat(A \conv C, P) &\simeq \Nat(A, \inthom(C,P)) \\
					&\simeq \Nat(T_nA, \inthom(C,P)) \\
					&{ \text{ \small(as } \inthom(C,P) \text{ is polynomial by \cref{internal_hom_remains_poly})}} \\
					&\simeq \Nat(T_nA \conv C,P)
\end{align*}
where in the second step we used that $\inthom(C,P)$ is polynomial by \cref{internal_hom_remains_poly}.

\end{proof}

\begin{proposition}
\cref{orthogonal_monoidality_theorem} also holds for functors valued in pointed spaces, equipped with the smash product. 
\end{proposition}

\begin{proof}[Proof sketch]

Recall the key steps of the proof
\begin{enumerate}
\item \cref{convolution-commutes-with-colimits}: Day Convolution commutes with colimits
\item \cref{internal_hom}: The internal hom functor exists.
\item \cref{internal_hom_remains_poly}: The internal hom into a $n$-polynomial functor is $n$-polynomial.
\item \cref{section_on_monoidal_localisation_theorem}: We check the compatibility condition of  \cref{compatible_localisation} to conclude via \cref{compatible_localisation_consequence} that the localisation is monoidal.
\end{enumerate}

They are all readily adapted. In the first step, we use that $(\csp, \wedge)$ is closed monoidal. The other steps are then quite formal and work the same as for $\cs$.

\end{proof}

The same proof cannot work for the categories of polynomial functors valued in pointed spaces equipped with the cartesian product, as the cartesian product in $\csp$ and thus the Day convolution product in $\fun(\cj,\csp)$ do not commute with colimits and hence cannot admit an internal hom functor. \\

However, the same result is still true.

\begin{proposition}
\cref{orthogonal_monoidality_theorem} also holds for functors valued in pointed spaces, equipped with the cartesian product.
\end{proposition}

\begin{proof} 

This follows relatively easily from \cref{orthogonal_monoidality_theorem}. Namely, consider the localisation
\[\begin{tikzcd}
	{\poly^{\leq n}(\cj,\csp)} && {\fun(\cj,\csp)}
	\arrow[""{name=0, anchor=center, inner sep=0}, "{T_n}"', curve={height=12pt}, from=1-3, to=1-1]
	\arrow[""{name=1, anchor=center, inner sep=0}, "\iota"', curve={height=12pt}, from=1-1, to=1-3]
	\arrow["\dashv"{anchor=center, rotate=-90}, draw=none, from=0, to=1]
\end{tikzcd}\]

By \cref{compatible_localisation_consequence}, just like in the proof of \cref{orthogonal_monoidality_theorem}, it is sufficient to check that for $F,G \in \fun(\cj,\csp)$, $T_n(\eta_F \conv \text{id}_G)$ is an equivalence (as a natural transformation of functors to pointed spaces). But as the cartesian product of pointed spaces is just the cartesian product of the corresponding spaces (equipped with the obvious basepoint), we have already shown in \cref{orthogonal_monoidality_theorem} that it is an equivalence of functors to spaces. To conclude, we remember that a map between pointed spaces which is an unpointed equivalence is a pointed equivalence.
\end{proof}

\chapter{Monoidality and Derivative Spectra}
\label{monoidal_oc_spectra}
Finally, we are ready to come back to the question that started this endeavour, namely

\begin{question}
Let $F \colon \cj \to \cs$ be a monoidal functor. To what structure on the $n$-th derivative spectra $\Theta^nF$ does the multiplication map
$$F \conv F \to F$$
correspond?
\end{question}

At the end of this chapter, we will have our answer.

\begin{answer}[\cref{final_derivative_spectra_result}]
Let $F \colon \cj \to \cs$ be a monoidal functor. Then, for any $n \in \mathbb{N}$, there are $O(n)$-equivariant maps
$$\Theta^nF \otimes \Theta^nF \to \Theta^nF \otimes \mathbb{S}^{\text{Ad}_{\text{O}(n)}}$$
\end{answer}

\section{Pointed and Unpointed Functors}
\label{section_pointed_and_unpointed_functors}
In this chapter, we mostly work with functors $F \colon \cj \to \csp$, i.e.\ functors to pointed spaces. 
When we form their convolution product and internal hom, the one of the corresponding functors to $\cs$ is meant. These a priori unpointed convolution products and internal homs are functors to pointed spaces as well (as we show now). This should be understood in analogy with spaces and pointed spaces: The product of pointed spaces as well as the  space of not necessarily basepoint-preserving maps of pointed spaces are pointed spaces, even though there is no adjunction between them.

\begin{proposition}
\label{internal_hom_becomes_pointed}
For functors $A,B \colon \cj \to \csp$, both their convolution product $A \conv B \colon \cj \to \cs$ and their internal hom $\inthom(A,B) \colon \cj \to \cs$ naturally lift to functors to pointed spaces.
\end{proposition}

\begin{proof}
First observe that a lift 
\[\begin{tikzcd}
	& \csp \\
	\cj & \cs
	\arrow["A", from=2-1, to=2-2]
	\arrow[from=1-2, to=2-2]
	\arrow[dotted, from=2-1, to=1-2]
\end{tikzcd}\]
corresponds to a natural transformation
$$ * \to A$$
 where $*$ denotes to constant functor with value the one-point space. This functor is the monoidal unit in $\fun(\cj,\cs)$ with respect to Day convolution.

The convolution product $A \conv B$ is then pointed via
$$* \simeq * \conv * \to A \conv B$$
{
More explicitly, recall that 
$$(A \conv B)(V) \simeq \colim_{V_1 \oplus V_2 \to V} A(V_1) \times B(V_2)$$
In this perspective, the basepoint is given by the class of the basepoint (i.e.\ product of the basepoints) in $A(0) \times B(0)$.
}

For the internal hom, observe that 
$$\Nat(*,\inthom(A,B)) \simeq \Nat(* \conv A,B) \simeq \Nat(A,B)$$
The basepoint in $\Nat(A,B)$ is given by the constant natural transformation
$$A \to * \to B$$
\end{proof}

\begin{proposition}[Monoids are canonically pointed]
\label{monoids_are_canonically_pointed}
Let $F \colon \cj \to \cs$ be lax symmetric monoidal. Then $F$ canonically lifts to a functor $F \colon \cj \to \csp$ to pointed spaces. The corresponding natural transformation $* \to F$ is a homomorphism of monoidal functors.
\end{proposition}

\begin{proof}
By definition, $F$ is equipped with a morphism $\varphi \colon 1_\cs \to F(1_\cj)$. As $1_\cs = \{*\}$ and $1_\cj = 0$, this defines a basepoint in $F(0)$. As $0 \in \cj$ is initial, this also defines a basepoint in any $F(V)$. The homomorphism property follows from the coherence axioms for lax monoidal functors.\\
Alternatively (but equivalently), such an $F$ is a commutative monoid in $\fun(\cj,\cs)$ and hence receives a map from the tensor unit of $\fun(\cj,\cs)$. In general, the tensor unit with respect to Day convolution in $\fun(\cC,\cD)$ is given by composing
$$\Map_\cC(1_\cC,-) \colon \cC \to \cs$$
with the unique cocontinuous functor
\begin{align*}
\cs &\to \cD \\
X &\mapsto (1_\cD)^{\oplus X}
\end{align*}
In our case, as $1_\cC = 1_\cj = 0$ is initial, the tensor unit in $\fun(\cj,\cs)$ is the constant functor with value $*$, which also proves the claim. 
\end{proof}

\begin{corollary}
Let $F,G \in \calg(\fun(\cj,\cs))$. Then $\inthom(F,G)$ is canonically pointed.
\end{corollary}
\begin{proof}
Apply \cref{internal_hom_becomes_pointed}.
\end{proof}

\section{Monoidality of Homogeneous Layers}
There are two corollaries left to be drawn from the earlier \cref{orthogonal_monoidality_theorem}, which are almost but not quite the same.

\begin{corollary}
\label{corollary_monoidal_homogeneous_layers_spaces}
For $F \colon \cj \to \cs$ lax symmetric monoidal, its homogeneous layers 
$$L_nF = \fib(T_nF \to T_{n-1}F) \colon \cj \to \cs$$
(where the fibres are taken over the canonical basepoints provided by the monoidal unit) are lax monoidal as well.
\end{corollary}

\begin{proof}
By \cref{orthogonal_monoidality_theorem}, $T_nF$ and $T_{n-1}F$ are monoidal. The natural transformation $T_nF \to T_{n-1}F$ is a morphism of monoidal functors, by commutativity of
\[\begin{tikzcd}
	{\fun(\cj,\cs)} & {\poly^{\leq n}(\cj,\cs)} \\
	& {\poly^{\leq n-1}(\cj,\cs)}
	\arrow["{{T_{n-1}}_{|\poly^{\leq n}(\cj,\cs)}}", from=1-2, to=2-2]
	\arrow["{T_{n-1}}"', from=1-1, to=2-2]
	\arrow[hook', from=1-2, to=1-1]
\end{tikzcd}\]
as the inclusion $\poly^{\leq n}(\cj,\cs) \hookrightarrow \fun(\cj,\cs)$ is lax monoidal.

 Hence, $L_nF$ is monoidal, as it is the kernel of a morphism of monoidal functors which preserves the unit. (See \cite[Cor.~3.2.2.5]{lurie-ha} for details.)
\end{proof}

\begin{corollary}
\label{corollary_monoidal_homogeneous_layers_pointed_spaces}
For $F \colon \cj \to (\cs_*,\times)$ lax symmetric monoidal, its homogeneous layers 
$$L_nF = \fib(T_nF \to T_{n-1}F) \colon \cj \to \cs_*$$ 
are lax monoidal as well.
\end{corollary}
\begin{proof}
By \cref{corollary_monoidal_homogeneous_layers_spaces}, they are monoidal functors to spaces. By \cref{monoids_are_canonically_pointed}, this is the same as being monoidal functors to pointed spaces (with the cartesian product).
\end{proof}

\section{Classification of Homogeneous Layers and the Internal Hom}
\sectionmark{Homogeneous Layers and Internal Hom}

The next step is to combine this with the classification of the homogeneous layers. Recall from \cref{theorem_classification_of_homogeneous_functors} that any $n$-homogeneous functor to pointed spaces is of the form

$$V \mapsto \Omega^\infty \left( \left( S^{V \otimes \mathbb{R}^n} \wedge \Phi\right)_{hO(n)} \right)$$

as there is an equivalence

$$\homog \simeq \spc^{\BO(n)} $$

Let's denote the involved functors by
\[\begin{tikzcd}
	\homog & {\spc^{\BO(n)}}
	\arrow["{\Theta^n}", shift left=2, from=1-1, to=1-2]
	\arrow["\Phi", shift left=2, from=1-2, to=1-1]
\end{tikzcd}\]
where $\Theta^n$ is the $n$-th derivative spectrum (over the basepoint) as constructed by Weiss and 
\begin{align*}
\Phi \colon \spc^{\BO(n)} &\to \homog \\
A &\mapsto \left( V \mapsto \Omega^\infty \left( (S^{nV} \wedge A)_{hO(n)} \right) \right)
\end{align*}

Homogeneous functors to $\csp$ are in particular functors to $\cs$, thus it makes sense to ask for a description of $\inthom(\Phi A,\Phi B)$ for $A,B \in \spc^{\BO(n)}$. 

\begin{proposition}
For  $A,B \in \spc^{\BO(n)}$ and $X \in \cj$
\begin{align*}
&\inthom(\Phi A,\Phi B)(X) \\
&\simeq \Nat_{\fun(\cj,\cs)}\left(\Omega^\infty \left( (S^{nV} \wedge A)_{hO(n)}\right), \Omega^\infty \left( (S^{nV} \wedge S^{nX} \wedge A)_{hO(n)} \right)  \right)
\end{align*}
\end{proposition}

\begin{proof}
A straightforward application of \cref{pointwise_formula_for_internal_hom} gives
\begin{align*}
&\inthom(\Phi A,\Phi B)(X) \\
&\simeq \Nat_{\fun(\cj,\cs)}\left(\Omega^\infty \left((S^{nV} \wedge A)_{hO(n)}\right), \Omega^\infty \left((S^{n(V \oplus X)} \wedge A)_{hO(n)} \right)   \right)
\end{align*}
Using $S^{n(V \oplus X)} \simeq S^{nV} \wedge S^{nX}$ yields the claim.
\end{proof}

\begin{remark}
\label{remark_theorem_classification_of_homogeneous_functors}
Recall that as we have $\homog \simeq \spc^{\BO(n)}$, we know that
\begin{align*}
\Nat_{\fun(\cj,\csp)}\left(\Omega^\infty \left( (S^{nV} \wedge A)_{hO(n)}\right), \Omega^\infty \left( (S^{nV} \wedge S^{nX} \wedge B)_{hO(n)} \right) \right) \\
\simeq \Map_{\spc^{\BO(n)}}(A,S^{nX} \wedge B)
\end{align*}

\end{remark}
\begin{definition}
For $A,B \colon \cj \to \csp$, denote by $\inthom_*(A,B)$ the subfunctor of $\inthom(A,B)$ given by pointed natural transformations, i.e.\
$$\inthom_*(A,B)(V) \simeq \Nat_{\fun(\cj,\csp)}(A,\sh_V^*B) $$
(This is the internal hom functor in $\fun((\cj,\oplus),(\csp,\wedge))$.)
\end{definition}
\begin{proposition}
\label{pointed_unpointed_hom_splitting}
For functors $A,B \colon \cj \to \cs_*$, there is a natural fibre sequence
$$\inthom_*(A,B) \to \inthom(A,B) \to B $$
which splits and hence induces an equivalence
$$\Theta^n \inthom_*(A,B) \oplus \Theta^n B \simeq \Theta^n \inthom(A,B)$$
\end{proposition}

\begin{proof}
The second map is given by evaluation at the basepoint of $A$. More formally, the pointing of $A$ is given by a map $p_A \colon * \to A$. This induces the map 
$$p_A^* \colon \inthom(A,B) \to \inthom(*,B) \simeq B$$
The fibre of this map over the basepoint of $B$ is then given by the basepoint-preserving transformations, i.e.\ $\inthom_*(A,B)$.
The splitting is given by
$$(A \to *)^* \colon \inthom(*,B) \to \inthom(A,B)$$
(i.e.\ by sending every $a \in A(V)$ to the same chosen $b \in B(V)$ for all $V$ simultaneously). \\
The formula for $\Theta^n$ then follows as $\Theta^n$ preserves finite limits, as $T_n$ and $T_{n-1}$ do and $L_n$, which determines $\Theta^n$, is the fibre of $T_n \to T_{n-1}$.
\end{proof}

\subsection{Calculation of $\Theta^n \inthom_*(A,B)$ for $A,B$ $n$-homogeneous}
In this subsection, we will do the following calculation.
\begin{proposition}
\label{derivative_of_pointed_internal_hom}
For $A,B \in \spc^{\BO(n)}$ there is an equivalence
\begin{align*}
\Theta^n \inthom_*(\Phi A, \Phi B) 
	&\simeq \Map_{\spc}(A,B) \otimes D_{O(n)}
\end{align*}
with $D_{O(n)} \simeq (\Sigma^\infty_+ O(n))^{hO(n)}$ being the dualising spectrum of $O(n)$ (see \cref{recall_dualising_spectrum}). 
\end{proposition}

On the right-hand side, $\Map_{\spc}(A,B) \otimes D_{O(n)}$, $O(n)$ acts diagonally, i.e.\ via the natural action on $D_{O(n)}$ and the conjugation action on $\Map_{\spc}(A,B)$.
The proof is given at the end of this section, after some preparatory calculations.
The observation that gets us going is that
$$\Map_{\spc^{\BO(n)}}(A,B) \simeq (\Map_{\spc}(A,B))^{hO(n)}$$
where $O(n)$ acts on $\Map_{\spc}(A,B)$ by conjugation. We shall first consider the functor
$$V \mapsto \Map(A,S^{nV} \wedge B)$$
and then take care of the homotopy fixed point construction.
\begin{lemma}
\label{free_homogeneous_functor}
For any (nonequivariant) spectrum $A$, the functor
\begin{align*}
\cj &\to \csp \\
V &\mapsto \Omega^\infty (S^{nV} \wedge A)
\end{align*}
is $n$-homogeneous, with $n$-th derivative spectrum given by 
$$\Theta^n (\Omega^\infty (S^{nV} \wedge A)) \simeq A \wedge O(n)_+$$
with the $O(n)$-action being given by the canonical left action on $O(n)_+$.
\end{lemma}
\begin{proof}
We can associate to any spectrum $X$ the $O(n)$-free spectrum 
$$
 O(n)_+ \wedge X$$
It satisfies
$$(O(n)_+ \wedge X)_{hO(n)} \simeq X$$
Applying this to $X=S^{nV} \wedge A$, we obtain
$$\Omega^\infty (S^{nV} \wedge A) \simeq \Omega^\infty \left( (S^{nV} \wedge (O(n)_+ \wedge A))_{hO(n)} \right)$$
To conclude that $O(n)_+ \wedge A$ is indeed the derivative, we implicitly use that there is a shearing isomorphism which identifies the $O(n)$-action acting only on $O(n)_+$ with the one acting diagonally on $O(n)_+$ and on the $\mathbb{R}^n$ in $S^{nV} \simeq S^{\mathbb{R}^n \otimes V}$.
\end{proof}
\begin{corollary}
\label{derivative_of_nonequivariant_mapping_spectrum_functor}
For $A,B \colon \cj \to \csp$, the functor
\begin{align*}
\cj &\to \csp \\
V &\mapsto \Map_{\spc}(\Theta^nA, \Theta^n (\sh_V^* B)) \simeq S^{nV} \wedge \Map_{\spc}(\Theta^nA, \Theta^n B)
\end{align*}
is $n$-homogeneous with $n$-th derivative spectrum given by
$$O(n)_+ \wedge \Map_{\spc}(\Theta^nA, \Theta^n B)$$
The $O(n)$-structure of this derivative spectrum is given by the action only on the $O(n)_+$-factor.
\end{corollary}
\begin{proof}
This follows from \cref{free_homogeneous_functor} by setting $A = \Map_{\spc}(\Theta^nA,\Theta^nB)$.
\end{proof}
Now we want to make use of the $O(n)$-action on $\Map_{\spc}(\Theta^nA,\Theta^nB)$ and identify $\Theta^n (\Map_{\spc}(\Theta^nA,S^{nV} \wedge \Theta^nB))^{hO(n)}$.
\begin{lemma}
Let $F \colon \cj \to \csp$ be an $n$-polynomial functor with an action of a (topological) group $K$. Then $F^{hK}$ is $n$-polynomial.
\end{lemma}
\begin{proof}
Being $n$-polynomial is defined by a certain map into a limit being an equivalence. Since homotopy fixed points are a limit construction, they commute with this limit. \\
$F$ being $n$-polynomial means that
$$F(V) \to \lim_{0 \neq U \subseteq \mathbb{R}^{n+1}}  F(V \oplus U)$$
is an equivalence. Thus, for $F^{hK}$, we have
\[\begin{tikzcd}
	{F(V)^{hK}} && {\lim_{0 \neq U \subseteq \mathbb{R}^{n+1}} (F(V \oplus U)^{hK})} \\
	&& {(\lim_{0 \neq U \subseteq \mathbb{R}^{n+1}} F(V \oplus U))^{hK}}
	\arrow["\simeq", from=1-3, to=2-3]
	\arrow["\simeq"', from=1-1, to=2-3]
	\arrow[from=1-1, to=1-3]
\end{tikzcd}\]
as $(-)^{hK}$ is a limit and thus commutes with limits. The diagonal map is an equivalence since it is the image of one under the functor $(-)^{hK}$, hence so is the third one, and $F^{hK}$ is $n$-polynomial.
\end{proof}
We can calculate $\Theta^n (F^{hK})$ in the following way.
\begin{lemma}
\label{derivative_of_fixed_points_of_homogeneous_functor}
Let $F \colon \cj \to \cs_*$ be $n$-homogeneous and let $K$ be a (topological) group acting on $F$. Then $\Theta^n (F^{hK}) \simeq (\Theta^n F)^{hK}$.
\end{lemma}
\begin{proof}
For this, we have to make heavier use of the methods of Weiss' original paper than before. Recall in particular from his approach:
\begin{itemize}
\item There are categories $\cj_0 \subset \ldots \subset \cj_n$ enriched in pointed spaces, and $F^{(n)} \colon \cj_{n} \to \cs$, the \enquote{unstable $n$-th derivative} of $F$, is defined as a Kan extension of $F \colon \cj_0 \to \cs$ along $\cj_0  \to \cj_{n}$, and $\Map_{\cj_0}(V,W) = \Map_\cj(V,W)_+$. (This is in \cite[Section~2]{weiss-oc}¨.)
\item $F^{(n)}$ relates to the derivative spectra $\Theta^nF$ by $(\Theta^nF)_{nl} = F^{(n)}(\mathbb{R}^l)$. (See e.g.\ \cite{weiss-oc}, p.2, p.3, p.10, p.12.)
\item When $F$ is $n$-homogeneous, this presents $\Theta^nF$ as an $\Omega$-spectrum. (See \cite[Cor.~5.12]{weiss-oc}.)
\end{itemize}
By \cite[Prop.~5.3]{weiss-oc} there are fibre sequences
$$F^{(n)}(V) \to F(V) \to \lim_{0 \neq U \subseteq \mathbb{R}^{n}} F(U \oplus V)$$
Since $(-)^{hK}$ is a limit these induce fibre sequences
$$(F^{(n)}(V))^{hK} \to F(V)^{hK} \to \left( \lim_{0 \neq U \subseteq \mathbb{R}^{n}} F(U \oplus V)\right)^{hK}$$
As limits commute with limits,
$$\left( \lim_{0 \neq U \subseteq \mathbb{R}^{n}} F(U \oplus V)\right)^{hK} \simeq \lim_{0 \neq U \subseteq \mathbb{R}^{n}} \left(F(U \oplus V)^{hK}\right) $$
Thus there are fibre sequences
$$(F^{(n)}(V))^{hK} \to F(V)^{hK} \to \left(\lim_{U \in G_{n}} F(U \oplus V)^{hK}\right)$$
whence
$$\left(F^{(n)}\right)^{hK} \simeq \left(F^{hK}\right)^{(n)}$$
Since $(\Theta^nF)_{nl} = F^{(n)}(\mathbb{R}^l)$ describes an $\Omega$-spectrum, $(-)^{hK}$ (as a limit) can be computed levelwise.
\\ Hence
$$\left( \Theta^n \left(F^{hK} \right) \right)_{nl} \simeq \left(F^{hK}\right)^{(n)}(\mathbb{R}^l) \simeq \left(F^{(n)}\right)^{hK}(\mathbb{R}^l) \simeq  \left( \Theta^n \left(F \right)^{hK} \right)_{nl}$$
\end{proof}
\begin{corollary}
For $A,B \in \spc^{\BO(n)}$, the $n$-th derivative spectrum of the functor
\begin{align*}
V \mapsto \Map_{\spc^{\BO(n)}}(A,S^{nV} \wedge B)
\end{align*}
is given by
\begin{align*}
(O(n)_+ \wedge \Map_{\spc}(A,B))^{hO(n)}
\end{align*}
The $O(n)$-homotopy fixed points are taken with respect to the $O(n)$-action from the right on $O(n)_+$ and the conjugation action on $\Map_{\spc}(A,B)$. \\
As an $n$-th derivative spectrum, $(O(n)_+ \wedge \Map_{\spc}(A,B))^{hO(n)}$ is equipped with another $O(n)$-action. This action is given by left multiplication on $O(n)_+$.
\end{corollary}
\begin{proof}
This follows from \cref{derivative_of_nonequivariant_mapping_spectrum_functor} and \cref{derivative_of_fixed_points_of_homogeneous_functor}. (We also explain the identification of the actions once again in \cref{spelling_out_the_actions}.)
\end{proof}
As we soon want to replace a homotopy fixed point construction by a homotopy orbit construction, the Norm map will play a role. Recall from algebra that for $G$ a finite group and $M$ a $G$-module, the Norm map is given by
\begin{align*}
\text{Nm} \colon M_G  &\to M^G \\
\overline{m} &\mapsto \sum_{g \in G} gm
\end{align*}
There is a generalisation of this to the case where $M$ is a spectrum and $G$ is a topological group, of which we will recall some basics now.
\begin{remark}[The Dualising Spectrum]
\label{recall_dualising_spectrum}
 $ $ \\
Recall that the Spivak-Klein dualising spectrum of a topological group $G$ is the $G$-spectrum $D_G \in \spc^{\text{BG}}$ given by $D_G = (\Sigma_+^\infty G)^{hG} = (\mathbb{S}[G])^{hG}$, where the homotopy fixed points are taken with respect to the left action of $G$ on itself. This spectrum has a residual right $G$-action. (Compare \cite{klein-dualising-spectrum-i}, Definition 1 for the original reference, or \cite{nikolaus-scholze} Definition I.4.2 for a more recent, albeit brief, treatment\footnote{
Nikolaus and Scholze would use the notation $D_{\text{BG}}$ for what Klein would call $D_G$.
}.).
The important properties of $D_G$ and in particular $D_{O(n)}$ as they pertain to this work are:
\begin{itemize}
\item For any $E \in \spc^{\text{BG}}$, there is a natural map called the norm map
$$\text{Nm} \colon (D_{\text{BG}} \otimes E)_{hG} \to E^{hG}$$
(\cite[Section~3]{klein-dualising-spectrum-i}, \cite[Theorem~I.4.1~(iv)]{nikolaus-scholze}) \\
A nice construction of this map is given as Remark 3.1 in \cite{klein-dualising-spectrum-i}, namely as the following composition pairing in $\spc^{\text{BG}}$, where $\text{hom}$ denotes the internal hom functor in $\spc^{\text{BG}}$:
$$\text{hom}(\mathbb{S},\mathbb{S}[G]) \otimes_{\mathbb{S}[G]} \text{hom}(\mathbb{S}[G],E) \to \text{hom}(\mathbb{S},E) $$
To make sense of this, one simply needs to realise that $\mathbb{S}[G] = \Sigma_+^\infty G$, $\text{hom}(\mathbb{S},X) \simeq X^{hG}$ and $A \otimes_{\mathbb{S}[G]} B \simeq (A \otimes B)_{hG}$ (with respect to the diagonal action of $G$).
\item The cofibre of the norm map is denoted $E^{tG}$ and called the Tate construction. Together, this forms the famous Tate fibre sequence
$$(D_{\text{G}} \otimes E)_{hG} \to E^{hG} \to E^{tG}$$
$(-)^{tG}$ is universally characterised by vanishing on compact objects (\cite{nikolaus-scholze} I.4.1 (iii)) such that the fibre of $E^{hG} \to E^{tG}$ commutes with colimits. 
\item 
When $G$ is a compact Lie group such as $O(n)$ and $E$ is an induced spectrum, i.e.\ of the form
$$E \simeq E' \otimes \Sigma^\infty_+ G$$
for a non-equivariant spectrum $E'$, then
$$E^{tG} \simeq 0$$
and equivalently
$$(D_{\text{BG}} \otimes E)_{hG} \simeq E^{hG}$$
(This is \cite[Theorem D, Cor.~10.2]{klein-dualising-spectrum-i}, or alternatively follows by identifying the compact objects of $\spc^{\text{BG}}$.)
\item (\cite[Theorem~10.1]{klein-dualising-spectrum-i}) For any compact Lie group $G$, there is an (equivariant) equivalence
$$D_G \simeq S^{\text{Ad}_G}$$
where $S^{\text{Ad}_G}$ denotes the suspension spectrum of the one point compactification of the adjoint representation $\text{Ad}_G \colon G \to \text{GL}(\mathfrak{g})$ of $G$ on its Lie algebra $\mathfrak{g}$.
\item We observe that $S^{\text{Ad}_{O(n)}}$ is invertible as an $O(n)$-spectrum, with inverse $S^{-\text{Ad}_{O(n)}}$. To see that, interpret
 $S^{\text{Ad}_{O(n)}}$ as the Thom spectrum of a vector bundle with fibre $\mathfrak{g}$ on $BO(n)$, take the inverse $-\text{Ad}_{O(n)} \in \text{KO}(BO(n))$ and construct its (virtual) Thom spectrum $S^{-\text{Ad}_{O(n)}}$. Then $S^{\text{Ad}_{O(n)}} \otimes S^{-\text{Ad}_{O(n)}} \simeq S^{\text{Ad}_{O(n)} + (- \text{Ad}_{O(n)})} \simeq S^0 = \mathbb{S}$.
\end{itemize}
\end{remark}
\begin{lemma}
\label{consequence_of_tate_fibre_sequence}
For $A,B \in \spc$, there is an equivalence
$$((\Map_{\spc}(A,B) \wedge O(n)_+))^{hO(n)} \simeq \Map_{\spc}(A,B) \otimes D_{O(n)}$$
where the homotopy fixed points are taken with respect to $O(n)$ acting on $O(n)_+$.
\end{lemma}
\begin{proof}
The Norm map for the $O(n)$-spectrum 
$$E = \Map_{\spc}(A,B) \wedge O(n)_+ $$ 
where $O(n)$ acts on $O(n)_+$ takes the form
\[\begin{tikzcd}
	{(D_{O(n)} \otimes \Map_{\spc}(A,B) \wedge O(n)_+)_{hO(n)}} & {(\Map_{\spc}(A,B) \wedge O(n)_+)^{hO(n)}} \\
	{}
	\arrow["{\text{Nm}}", from=1-1, to=1-2]
\end{tikzcd}\]
Since the action is free, $\text{Nm}$ is an equivalence (\cite[Cor.~10.2]{klein-dualising-spectrum-i}). The left-hand side simplifies as
$$D_{O(n)} \otimes \Map_{\spc}(A,B) \simeq {(D_{O(n)} \otimes \Map_{\spc}(A,B) \wedge O(n)_+)_{hO(n)}}$$
\end{proof}

At last, we can now calculate $\Theta^n \inthom_*(\Phi A, \Phi B)$:
\begin{proof}[Proof of \cref{derivative_of_pointed_internal_hom}]
The proof is now a straightforward combination of the previous calculations:
\begin{itemize}
\item $\inthom_*(\Phi A, \Phi B)(V) \simeq \Map_{\spc^{\BO(n)}}(A, S^{nV} \wedge B)$ by definition and classification of pointed homogeneous functors.
\item $\Map_{\spc^{\BO(n)}}(A, S^{nV} \wedge B) \simeq (\Map_{\spc}(A, S^{nV} \wedge B))^{hO(n)}$ with $O(n)$ acting by conjugation.
\item Ignoring the $(-)^{hO(n)}$ for a moment, the functor 
$$V \mapsto \Map_{\spc}(A, S^{nV} \wedge B)$$ 
is $n$-homogeneous with $n$-th derivative 
$$\Map_{\spc}(A,B) \wedge O(n)_+$$
by \cref{derivative_of_nonequivariant_mapping_spectrum_functor}.
\item Note that $\Map_{\spc}(A,B) \wedge O(n)_+$ now has two $O(n)$-actions: One acting by multiplication from the left on $O(n)_+$ (this is the action making the derivative spectrum an $O(n)$-spectrum), and a second one acting on $\Map_{\spc}(A,B)$ by \enquote{conjugation} and on $O(n)_+$ from the right by multiplication with the inverse. This second action is the one coming from the action on $\Map_{\spc}(A,\sh_V^* B)$ under the identifications we made. Note that we spell them out once more in \cref{spelling_out_the_actions}.
By \cref{derivative_of_fixed_points_of_homogeneous_functor}, the homotopy fixed point spectrum with respect to the second action calculates 
$$\Theta^n \Map_{\spc^{\BO(n)}}(A, S^{nV}B) \simeq (\Map_{\spc}(A,B) \wedge O(n)_+)^{hO(n)}$$
\item By \cref{consequence_of_tate_fibre_sequence}, after equivariantly identifying 
$\Map_{\spc}(A,B) \wedge O(n)_+$ with the diagonal $O(n)$-action on both factors with the same spectrum with $O(n)$ acting only on the second factor\footnote{
via a shearing isomorphism of the form $A \otimes G \to A \otimes G, (a,g) \mapsto (ga,g)$
},
this is equivalent to 
$$D_{O(n)} \otimes \Map_{\spc}(A,B)$$
with the remaining $O(n)$-action given diagonally.
\end{itemize}
\end{proof}

\begin{remark}[Spelling out the actions once more] \label{spelling_out_the_actions}
$ $ \\
The identification
\begin{align*}
S^{nV} \wedge \Map_{\spc}(A,B) &\simeq (S^{nV} \wedge \Map_{\spc}(A,B) \wedge O(n)_+)_{hO(n)} \\
(s,f) &\mapsto [(s,f,1)] \\
(l^{-1}s,f) &\mapsfrom [(s,f,l)]
\end{align*}

identifies the $O(n)$-action on the left side which is diagonal on $S^{nV}$ and $\Map_{\spc}(A,B)$, i.e.\ given by $(s,f) \mapsto (ks,kf)$ for $k \in O(n)$ with the one on $\Map_{\spc}(A,B) \wedge O(n)_+$ which is the same action on $\Map_{\spc}(A,B)$ and given by acting with $k^{-1}$ from the right on $O(n)_+$.

See also

\[\begin{tikzcd}
	{(o^{-1}s,f)} & {[(s,f,o)]} \\
	{(ko^{-1}s,kf)} & {[(ko^{-1}s,kf,1)] = [(s,kf,ok^{-1})]}
	\arrow[maps to, from=1-2, to=1-1]
	\arrow[maps to, from=1-1, to=2-1]
	\arrow[maps to, from=1-2, to=2-2]
	\arrow[maps to, from=2-2, to=2-1]
\end{tikzcd}\]

Hence on
$$\Theta^n \Map_{\spc^{\BO(n)}}(A,B) \simeq \Map_{\spc}(A,B) \wedge O(n)_+$$
there are two commuting $O(n)$-actions:
\begin{itemize}
\item The \enquote{inner} action, i.e.\ the one present on any $n$-th derivative spectrum, acts from the left on $O(n)_+$.
\item The \enquote{outer} action, i.e.\ the one coming from the conjugation action on $\Map_{\spc}(A,B)$, acts in the usual way on  $\Map_{\spc}(A,B)$ and from the right on $O(n)_+$.
\end{itemize}
Then, under the identification
$$\left( D_{O(n)} \otimes \Map_{\spc}(A,B) \otimes O(n)_+ \right)_{hO(n)} \simeq \left( \Map_{\spc}(A,B) \wedge O(n)_+ \right)^{hO(n)}$$
(where the homotopy fixed points and orbits are taken w.r.t.\ the \enquote{outer} action), the inner action acts (only) on $O(n)_+$ on either side of this description. But then, under the final identification
$$ \left( D_{O(n)} \otimes \Map_{\spc}(A,B) \otimes O(n)_+ \right)_{hO(n)} \simeq D_{O(n)} \otimes \Map_{\spc}(A,B)$$ 
the \enquote{inner} $O(n)$-action on the left hand side is identified with the $O(n)$-action on the right-hand side on both tensor factors, for the same reason as in the first identification in this remark.
\end{remark}

\subsection{Calculation of $\Theta^n (A \conv B)$ for $A,B$ $n$-homogeneous}

After calculating $\Theta^n \inthom_*(A,B)$ for $n$-homogeneous functors $A$ and $B$, we also calculate $\Theta^n(A \conv B)$ for $A,B$ $n$-homogeneous. For our final result, either this or the calculation of $\Theta^n \inthom_*(A,B)$ are sufficient on their own, but it is nice to have both.

\begin{proposition}
\label{calculation_of_theta_n_A_conv_B}
For $A,B \colon \cj \to \csp$ $n$-homogeneous, there is an equivalence

$$\Theta^n(A \conv B) \simeq \Theta^n A \oplus \Theta^nB \oplus D_{O(n)}^\vee \otimes \Theta^n A \otimes \Theta^n B$$
where $D_{O(n)}^\vee \simeq S^{-\text{Ad}_n}$ denotes the dual $O(n)$-spectrum to $D_{O(n)}$.
\end{proposition}

The proof consists of two steps:

\begin{lemma}
\label{lemma_splitting_off_A_and_B}
For functors $A,B \colon \cj \to \csp$, there is a cofibre sequence
$$A \vee B \to A \conv B \to A \redconv B$$
where $(A \vee B)(V) := A(V) \vee B(V)$ and $A \redconv B$ denotes the \textit{reduced} convolution product of $A$ and $B$, taken with respect to the smash product monoidal structure on $\csp$ instead of the cartesian product monoidal structure. \\
At the level of derivative spectra, the sequence splits and thus induces an equivalence
$$\Theta^n (A \conv B) \simeq \Theta^n A \oplus \Theta^n B \oplus \Theta^n (A \redconv B)$$
\end{lemma}

\begin{lemma}
\label{lemma_derivative_of_reduced_convolution}
For $A,B \colon \cj \to \csp$ $n$-homogeneous, there is an equivalence of $O(n)$-spectra 
$$\Theta^n (A \redconv B) \simeq \Theta^n A \otimes \Theta^n B \otimes D_{O(n)}^\vee$$
\end{lemma}

\begin{proof}[Proof of \cref{calculation_of_theta_n_A_conv_B}] \leavevmode \\
This is immediate from \cref{lemma_splitting_off_A_and_B} and \cref{lemma_derivative_of_reduced_convolution}. \qedhere
\end{proof}
Now, it just remains to prove the two lemmas.

\begin{proof}[Proof of \cref{lemma_splitting_off_A_and_B}]
On $V \in \cj$, the sequence
$$(A \vee B)(V) \to (A \conv B)(V) \to (A \redconv B)(V)$$
takes the form
$${\colim_{U_1 \oplus U_2 \to V}A(U_1) \vee B(U_2)} \to {\colim_{U_1 \oplus U_2 \to V}A(U_1) \times B(U_2)} \to {\colim_{U_1 \oplus U_2 \to V}A(U_1) \wedge B(U_2)}$$
after using 
$${\colim_{U_1 \oplus U_2 \to V}A(U_1) \vee B(U_2)} \simeq A(V) \vee B(V)$$
Hence it is a cofibre sequence since the sequences
$$A(U_1) \vee B(U_2) \to A(U_1) \times B(U_2) \to A(U_1) \wedge B(U_2)$$
are cofibre sequences.

On the level of derivative spectra, the required splitting is induced by

\[\begin{tikzcd}
	{A \conv B} && {A \times B} \\
	{A \conv * \vee * \conv B} \\
	{A \vee B}
	\arrow["\simeq"{description}, from=3-1, to=2-1]
	\arrow[from=2-1, to=1-1]
	\arrow[from=1-1, to=1-3]
	\arrow[from=3-1, to=1-3]
\end{tikzcd}\]

where the map $A \conv B \to A \times B$ is the product of $A \conv B \to A \conv *$ and $A \conv B \to * \conv B$ (and we are using the stability of $\spc^{\BO(n)}$).

\end{proof}

\begin{proof}[Proof of \cref{lemma_derivative_of_reduced_convolution}] $ $ \\
Let $A = \Phi M$, $B=\Phi N$ for $M, N \in \spc^{\BO(n)}$. Recall here that we denote by $\Phi$ the inverse equivalence to $\homog \to \spc^{\BO(n)}$.

Then, by Yoneda, the following calculation for arbitrary $C \in \spc^{\BO(n)}$ shows that 
$$\Theta^n(A \redconv B) \simeq D_{O(n)}^\vee \otimes M \otimes N$$
\begin{align*}
\Map_{\spc^{\BO(n)}}(\Theta^n(A \redconv B),C)
& \simeq \Nat_*(L_n(\Phi A \redconv \Phi B), \Phi C) \\
& \simeq \Nat_*(T_n(\Phi A \redconv \Phi B), \Phi C) \\
& \simeq \Nat_*(\Phi A \redconv \Phi B,\Phi C) \\
& \simeq \Nat_*(\Phi A,\inthom_*(\Phi B,\Phi C)) \\
& \simeq \Nat_*(\Phi A,L_n \inthom_*(\Phi B,\Phi C) \\
& \simeq \Nat_*(\Phi A,\Phi(\Map_{\spc}(B,C)\otimes D_{O(n)})) \\
& \simeq \Map_{\spc^{\BO(n)}}(M,\Map_{\spc}(N,C) \otimes D_{O(n)}) \\
& \simeq \Map_{\spc^{\BO(n)}}(D_{O(n)}^\vee \otimes M,\Map_{\spc}(N,C)) \\
& \simeq \Map_{\spc^{\BO(n)}}(D_{O(n)}^\vee \otimes M \otimes N,C)
\end{align*}

Almost all of the steps are straightforward applications of the equivalence $\homog \to \spc^{\BO(n)}$ and the adjunctions for $T_n \colon \fun(\cj,\csp) \to \poly^{\leq n}(\cj,\csp)$ and ${L_n}_{|\poly^{\leq n}(\cj,\csp)} \colon \poly^{\leq n}(\cj,\csp) \to \homog $. For the second step, we use that
$$T_n(\Phi A \redconv \Phi B) \simeq L_n(\Phi A \redconv \Phi B)$$
which holds as $T_{n-1} (\Phi A \redconv \Phi B) \simeq *$, which follows from the pullback square
\[\begin{tikzcd}
	{\poly^{\leq n}(\cj,\csp)} & \homog \\
	{\poly^{\leq n-1}(\cj,\csp)} & {*}
	\arrow[from=1-1, to=2-1]
	\arrow[from=1-2, to=1-1]
	\arrow[from=2-2, to=2-1]
	\arrow[from=1-2, to=2-2]
	\arrow["\lrcorner"{anchor=center, pos=0.125, rotate=-90}, draw=none, from=1-2, to=2-1]
\end{tikzcd}\]
and monoidality of $T_{n-1}$.

\end{proof}

\section{Final Result on Derivative Spectra of Monoidal Functors}
\sectionmark{Derivative Spectra of Monoidal Functors}

One of the questions which started this project was
\begin{center}
For a monoidal functor $F \colon \cj \to \cs$, what can be said about its $n$-th derivative spectra $\Theta^nF$?
\end{center}

Using \cref{derivative_of_pointed_internal_hom}, we can finally give the following answer

\begin{corollary}
\label{final_derivative_spectra_result}
Let $F \colon \cj \to \cs$ be a monoidal functor. Then, for any $n \in \mathbb{N}$, there are $O(n)$-equivariant maps
$$\Theta^nF \otimes \Theta^nF \to \Theta^nF \otimes D_{O(n)}$$
\end{corollary}

\begin{proof}
Fix $n \in \mathbb{N}$. Since $F$ is monoidal, it lifts to pointed spaces by \cref{monoids_are_canonically_pointed}. By our main theorem (\cref{orthogonal_monoidality_theorem}), $T_nF$ and $T_{n-1}F$ are monoidal. By \cref{corollary_monoidal_homogeneous_layers_spaces}, the $n$-homogeneous layer $L_nF$, which corresponds to $\Theta^n F$ under the classification of homogeneous functors (\cref{theorem_classification_of_homogeneous_functors}) is monoidal. Hence there is a structure map
$$L_nF \conv L_nF \to L_nF$$
Using the internal hom - adjunction and \cref{pointed_unpointed_hom_splitting}, we obtain
$$L_nF \to \inthom(L_nF,L_nF) \to \inthom_*(L_nF,L_nF)$$
Applying $\Theta^n$ to this and using \cref{derivative_of_pointed_internal_hom}, we get a map
$$\Theta^n F \to \Map_{\spc}(\Theta^nF,\Theta^nF) \otimes D_{O(n)} $$
As we pointed out in \cref{recall_dualising_spectrum} (last point), $D_{O(n)} \simeq S^{\text{Ad}_n}$ has an inverse $D_{O(n)}^\vee \simeq S^{-\text{Ad}_n}$. Thus this map is equivalent to a map
$$ \Theta^n F \otimes D_{O(n)}^\vee \to \Map_{\spc}(\Theta^nF, \Theta^nF) $$
By the internal hom adjunction in $\spc^{\BO(n)}$, this corresponds to a map
$$\Theta^nF \otimes D_{O(n)}^\vee \otimes \Theta^nF \to \Theta^n F$$
which corresponds to the desired map
$$\Theta^n F \otimes \Theta^nF \to \Theta^nF \otimes D_{O(n)}$$
\end{proof}

\begin{remark}
\label{remark_twisted_multiplication_structure}
These maps on derivative spectra will satisfy additional consistencies such as homotopies making 
\[\begin{tikzcd}
	{\Theta^n F \otimes \Theta^nF \otimes \Theta^nF} & {\Theta^nF \otimes \Theta^nF \otimes D_{O(n)}} \\
	{\Theta^nF \otimes D_{O(n)} \otimes \Theta^nF} & {\Theta^n F \otimes D_{O(n)}^{\otimes 2}}
	\arrow[from=1-1, to=1-2]
	\arrow[from=1-1, to=2-1]
	\arrow[from=2-1, to=2-2]
	\arrow[from=1-2, to=2-2]
\end{tikzcd}\]
commute, since the map
$$m \colon F \conv F \to F$$
comes with homotopies making the square
\[\begin{tikzcd}
	{F \conv F \conv F} & {F \conv F} \\
	{F \conv F} & F
	\arrow["{m \conv id}", from=1-1, to=2-1]
	\arrow["m", from=2-1, to=2-2]
	\arrow["m"', from=1-2, to=2-2]
	\arrow["{id \conv m}"', from=1-1, to=1-2]
\end{tikzcd}\]
commute.
\end{remark}

\begin{conjecture}
\label{corollary_untwisted_multiplication_structure}
For $F \colon \cj \to \csp$ lax symmetric monoidal, these maps make $\widetilde{\Theta^n}F := \Theta^nF \otimes D_{O(n)}^\vee$ into a nonunital $E_\infty$-ring spectrum.
\end{conjecture}
\begin{proposition}
There is a (non-equivariant) equivalence
$$ \widetilde{\Theta^n}F \simeq \Omega^\frac{n(n-1)}{2} \Theta^n F$$
\end{proposition}

\begin{proof}
As 
$$\dim O(n) = \frac{n(n-1)}{2}$$ we have (non-equivariantly) 
$$D_{O(n)} \simeq \Sigma^\infty S^{\text{Ad}_n} \simeq \mathbb{S}^{\frac{n(n-1)}{2}}$$ 
thus $$D_{O(n)}^\vee \simeq \mathbb{S}^{-\frac{n(n-1)}{2}}$$
thus 
$$\widetilde{\Theta^n}F \simeq D_{O(n)}^\vee \otimes \Theta^n F \simeq \Omega^{-\frac{n(n-1)}{2}} \Theta^n F$$
\end{proof}

\cref{corollary_untwisted_multiplication_structure} would then imply

\begin{conjecture}
\label{corollary_multiplication_on_theta_1}
For $F \colon \cj \to \cs$ lax symmetric monoidal, $\Theta^1 F$ is a nonunital $E_\infty$-ring spectrum.
\end{conjecture}

\begin{proof}
As $O(1) \cong \mathbb{Z}/2$ is $0$-dimensional, the dualising spectrum $D_{O(1)}$ is the sphere spectrum, hence $\widetilde{\Theta^1}F \simeq \Theta^1 F$.
\end{proof}

\chapter{Applications and Open Questions}
\label{chapter_applications_and_open_questions}

This short final chapter discusses the following topics:
\begin{itemize}

\item In \cref{section_examples_of_monoidal_functors}, examples of monoidal functors $\cj \to \csp$ which motivated this work and to which the main results, i.e.\ \cref{orthogonal_monoidality_theorem} and \cref{final_derivative_spectra_result}, can be applied, are discussed. However, as the proofs of the main results do not provide explicit formulas for the resulting structure on the polynomial approximations and on the derivative spectra, concrete formulas remain out of reach for now. 

\item In \cref{section_potential_alternative_description}, we explain a potential alternative description of the multiplication on $\Theta^1F$ of a monoidal functor $F$.

\item  In \cref{section_open_questions}, we present some open questions and challenges, as well as some directions for future work. These include more sophisticated examples of monoidal structures, the existence of a similar theorem in Goodwillie Calculus, as well as some more general open questions in Orthogonal Calculus which we find interesting.

\end{itemize}

\section{Examples of Monoidal Functors}
\label{section_examples_of_monoidal_functors}

\begin{example}
\label{monoidal_functor_examples}
The following functors $\cj \to \cs$ (already mentioned in \cref{example-functors}) are lax symmetric monoidal.
\begin{enumerate}[label=\alph*)]
\item 
$
\cj \to \cs \\
V \mapsto S^V \simeq V^c
$
\item
$
\cj \to \cs \\
V \mapsto O(V)
$
\item 
$
\cj \to \cs \\
V \mapsto \BO(V)
$
\item 
$
\cj \to \cs \\
V \mapsto \B \TOP(V)
$
\item 
$
\cj \to \cs \\
V \mapsto \B G(S^V)
$ \\

(Here, $G(X)$ denotes the grouplike topological monoid of homotopy automorphisms, i.e.\ self homotopy-equivalences, of $X$.)
\end{enumerate}
\end{example}

\begin{proof}
In all cases, we indicate the required maps
$$F(V) \times F(W) \to F(V \oplus W)$$
\begin{enumerate}[label=\alph*)]
\item 
For $V \mapsto S^V$, up to identifying $S^{V \oplus W}$ with $S^V \wedge S^W$, these are the quotient maps $S^V \times S^W \to S^V \wedge S^W \simeq (S^V \times S^W)/(S^V \vee S^W)$.

\item For $V \mapsto O(V)$, taking block sums of matrices/automorphisms defines maps 
\begin{align*}
O(V) \times O(W) &\to O(V \oplus W) \\
(A,B) &\mapsto \begin{pmatrix}
A & 0 \\
0 & B
\end{pmatrix} 
\end{align*}
\item For $V \mapsto \B O(V)$, apply the functor $\B$ to the maps for $O(V)$.
\item The cases $V \mapsto \TOP(V)$ and $V \mapsto \B \TOP (V)$ work analogously to $O(V)$ and $\BO(V)$.
\item For $V \mapsto G(S^V)$, the identification $S^V \wedge S^W \simeq S^{V \oplus W}$ gives maps
$$\Map(S^V,S^V) \times \Map(S^W,S^W) \to \Map(S^{V \oplus W},S^{V \oplus W})$$
which restrict to the desired maps. \qedhere
\end{enumerate}
\end{proof}

By
\cref{orthogonal_monoidality_theorem}
and
\cref{final_derivative_spectra_result}
the $n$-polynomial approximations of the functors of 
\cref{monoidal_functor_examples}
are lax symmetric monoidal, and their derivative spectra admit multiplication maps

$$\Theta^nF \otimes \Theta^nF \to \Theta^nF \otimes D_{O(n)}$$

We would like to be able to describe these maps explicitly, but are so far unable to do so. There are some guesses in easy cases.

\begin{example}
Let $F(V) = BO(V)$. By \cite{weiss-oc}, the first few derivative spectra are given by
\begin{align*}
\Theta^1 \BO &\simeq \mathbb{S} \\
\Theta^2 \BO &\simeq \Omega \mathbb{S} \\
\Theta^3 \BO &\simeq \Omega^2 \mathbb{S}/3
\end{align*}
The multiplication maps predicted by \cref{final_derivative_spectra_result} thus are maps of the form
\begin{itemize}
\item ($n=1$) $\mathbb{S} \otimes \mathbb{S} \to \mathbb{S}$
\item ($n=2$) $\mathbb{S}^{-2} \otimes \mathbb{S}^{-2} \to \mathbb{S}^{-2}$
\item ($n=3$) $\Omega^5 \mathbb{S}/3 \otimes \Omega^5 \mathbb{S}/3 \to \Omega^5 \mathbb{S}/3$
\end{itemize}

We conjecture that the first map is the usual multiplication map on the sphere spectrum. About the second, we remark that it must be non-equivariantly nullhomotopic, but we think that this is not true as an $O(2)$-equivariant map. \\
For general $n$, there is also the description of $\Theta^n \BO$ by Arone (\cite{arone-bo-bu}) which states that
$$ \Theta^n \BO \simeq \Map_*(L_n,D_{O(n)})$$
where $L_n$ is the unreduced suspension of the geometric realisation of the category of non-trivial direct-sum decompositions of $\mathbb{R}^n$. \\
A possible guess would be that the multiplication map is in general related to the diagonal map on $L_n$.

\end{example}

\begin{example}
\label{example_a_theory_as_derivative_of_btop}
Let $F(V)=\BTOP(V)$. Then
$$\Theta^1 F \simeq A(*) \simeq K(\mathbb{S}) \simeq \mathbb{S} \vee \text{Wh}^{\text{DIFF}}(*)$$
where $A(*)$ denotes Waldhausen
's A-theory spectrum of a point and $\text{Wh}^{\text{DIFF}}$ denotes the smooth Whitehead spectrum). (For a reference, see \cite{weiss-oc}, p.3, which references \cite[Theorem 2, Addendum 4]{waldhausen-manifold-approach} for this fact.) \\
By \cref{corollary_multiplication_on_theta_1} we obtain a nonunital multiplication on $A(*)$. We conjecture that it is the usual multiplication. A possible explanation for this fact uses work of Weiss-Williams which is yet to appear\footnote{
likely as \enquote{Automorphisms of Manifolds and Algebraic K-Theory IV}.
}.
\end{example}

\section{A Potential Alternative Description of the Multiplication On $\Theta^1F$}
\label{section_potential_alternative_description}
\sectionmark{Multiplication on $\Theta^1F$}

\begin{remark}[Potential alternative description of a multiplication on the first derivative spectrum] $ $ \\
\label{potential_explicit_formula}
There is a potential alternative explicit description of a multiplication on the first derivative spectrum of a monoidal functor using the explicit construction of the homotopy fibre of a map $f \colon (X,x_0) \to (Y,y_0)$ as
$$ \hofib f = \{(a,p) \in X \times PY \mid p(0)=f(a) \text{ and } p(1)=y_0\} $$
Recall that (e.g.\ as a special case of \cite[Prop.~5.3]{weiss-oc}) there is a fibre sequence
$$F^{(1)}(V) \to F(V) \to F(V \oplus \mathbb{R}) $$

Applying this we get
$$F^{(1)}(V) = \{(a,p) \in F(V) \times PF(V \oplus \mathbb{R}) \mid p(0)=i_*(a), p(1)=*\}$$
The monoidal structure on $F$ provides maps 
$$m_{V,W} \colon F(V) \times F(W) \to F(V \oplus W)$$
Using this one can also write down a multiplication map in the following way.

Define 
$$\tilde{m}_{V,W} \colon F^{(1)}(V) \times F^{(1)}(W) \to \Omega F^{(1)}(V \oplus W \oplus \mathbb{R})$$
by 
\begin{align*}
\tilde{m} ((a,p),(b,q)) = (\eta , \nu) &\in \Omega F^{(1)}(V \oplus W \oplus \mathbb{R}) \\
&\subset \Omega F(V \oplus W \oplus \mathbb{R}) \times \Omega P F(V \oplus W \oplus \mathbb{R} \oplus \mathbb{R})
\end{align*}
where
$$\eta = i_*(q) \ast m_{V \oplus \mathbb{R},W}(p,b) \ast \overline{m_{V,W \oplus \mathbb{R}}(a,q)} \ast i_*(\overline{p})$$
Here, $\overline{\gamma}$ for a path $\gamma$ denotes the reverse path.

The construction of $\nu$, which needs to be a path in $\Omega F^{(1)}(V \oplus W \oplus \mathbb{R} \oplus \mathbb{R})$, proceeds similarly. Concretely, if $s$ is the path-space variable, $\nu$ can be defined by
\begin{equation*}
v_s(t) = 
\begin{cases}
m(p,q)(s \cdot 4t,1) & \text{if } 0 \leq t \leq \frac{1}{4}\\
m(p,q)(s,s \cdot 4 (t-\frac{1}{4})) & \text{if } \frac{1}{4} < t \leq \frac{1}{2}\\
m(p,q)(s \cdot (1-4(t-\frac{1}{2})),s) & \text{if } \frac{1}{2} < t \leq \frac{3}{4}\\
m(p,q)(s(1-4(t-\frac{3}{4})),s) & \text{if } \frac{3}{4} < t \leq 1
\end{cases}
\end{equation*}

We have now described maps
$$F^{(1)}(V) \times F^{(1)}(W) \to \Omega F^{(1)}(V \oplus W \oplus \mathbb{R})$$

Since 
$$(\Theta^1 F)_n \simeq F^{(1)}(\mathbb{R}^n)$$
these maps together describe a map
$$m \colon \Theta^1 F \otimes \Theta^1 F \to \Theta^1 F \simeq \Sigma \Omega \Theta^1F$$
for a monoidal functor $F \colon \cj \to \csp$. \\
We conjecture that this map models the map 
$$\Theta^1 F \otimes \Theta^1F \to \Theta^1F$$ 
that is predicted by our theory (see \cref{corollary_multiplication_on_theta_1}). If so, this much more concrete perspective might enable concrete calculations. It seems poorly suited for establishing more abstract coherence properties though.
\end{remark}

\section{Further Open Questions}
\label{section_open_questions}

Here are some questions that we can't answer as of now but which might be interesting for further study.

\begin{question}[Rational $K$-Theory of $\mathbb{Z}$] $ $ \\
We already mentioned in \cref{example_a_theory_as_derivative_of_btop} that 
 $$\Theta^1 (V \mapsto \text{BTOP}(V)) \simeq A(*) \simeq K(\mathbb{S}) \simeq \mathbb{S} \vee \text{Wh}^{\text{DIFF}}(*)$$
Using $S(V)$ to denote the unit sphere in $V$, one can similarly show that 
$$\Theta^1 (V \mapsto \text{BTOP}(S(V))) \simeq \mathbb{S} \vee \Omega \text{Wh}^{\text{DIFF}}(*)$$

Rationally, the former is identified with $K(\mathbb{Z})$, the second with $\mathbb{S} \vee \Omega K(\mathbb{Z})$. Borel (\cite{borel}) proved that 
  \[
    \pi_* K(\mathbb{Z}) \otimes \mathbb{Q} = \left\{\begin{array}{lr}
    	\mathbb{Q}, & \text{for } *=0\\
        \mathbb{Q}, & \text{for } *=4k+1 \text{ with } k \geq 1\\
        0, & \text{otherwise}\\
        \end{array}\right.
  \]
so one sees that 
  \[
    \pi_* \Theta^1 (V \mapsto \text{BTOP}(S(V))) \otimes \mathbb{Q} = \left\{\begin{array}{lr}
    	\mathbb{Q}, & \text{for } *=0\\
        \mathbb{Q}, & \text{for } *=4k \text{ with } k \geq 1\\
        0, & \text{otherwise}\\
        \end{array}\right.
  \]
  
By virtue of the monoidality theorem for $\Theta^1$, this will carry a graded ring structure. It would now be very interesting to determine it. In particular, is this ring generated by the class in degree $4$?
\end{question}

\begin{question}
Although unexplained in this work, one possible purpose of Orthogonal Calculus is to generalise the theory of characteristic classes. E.g.\ for $F \colon \cj \to \cs$ one can construct a cohomology class on $F(\mathbb{R}^\infty)$ with values in $\Theta^1F$ which in the case of $V \mapsto \BO(V)$ recovers (up to change of coefficients) the classical total Stiefel-Whitney class (\cite{weiss-williams-iv}). Can the product structure on $\Theta^1F$ in case of a monoidal functor $F$ thus recover the classical product formula for the total Stiefel-Whitney class as well as prove generalisations of it, e.g.\ for Euclidean fibre bundles ($V \mapsto \BTOP(V)$) or spherical fibrations ($V \mapsto \text{BG}(V)$)?
\end{question}

\begin{question}[The gluing datum in the orthogonal tower and chromatic localization] $ $ \\
Even if one understands all the homogeneous layers of the orthogonal tower, the question 
how to reconstruct the $n$-polynomial approximations can remain difficult. In Goodwillie Calculus, the gluing data can be described as maps to certain Tate constructions, e.g.\ for functors to spectra, by a result of Kuhn (\cite{kuhn}, Prop.\ 5.13), there is a pullback diagram
\begin{center}
\begin{tikzcd}
T_d F \arrow[d] \arrow[r] & (\Delta_dF)^{h \Sigma_d} \arrow[d] \\
T_{d-1}F \arrow[r]        & (\Delta_dF)^{t \Sigma_d}          
\end{tikzcd}

\end{center}
In orthogonal calculus, a similar concrete description of the gluing data is still missing and should be investigated.

A fascinating consequence of the description via Tate constructions in Goodwillie Calculus is that, after certain chromatic localizations, i.e.\ $T(n)$-locally, Goodwillie towers split, as the Tate constructions are trivial (Kuhn, \cite[Thm.~7.6]{kuhn}).

Orthogonal calculus too can be done locally with respect to such localisations (see e.g.\ Taggart \cite{taggart-local}). An interesting follow-up question would therefore be if an analogue of Kuhn's theorem holds in orthogonal calculus too.
\end{question}

\begin{question}
\label{question_oc_over_more_general_fields}

In \cref{remark_oc_over_more_general_fields}, we observed that several functors of interest $\cj \to \cs$, such as $V \mapsto \BO(V)$, are also defined on a smaller category which does not know about inner products, only about injective maps and natural complements. 

This motivates the following question, which to our knowledge is open.

Let $K$ be a field, and let $\cj_K$ be the category of finite-dimensional $F$-vector spaces, with
$\mor_{\cj_K}(V,W)$ given by
$$\{f \colon V \to W, U_{V,W} \subseteq W \mid f \text{ linear and injective}, U_{V,W} \text{ a complement to \im(f)} \} $$
with composition maps
$$\mor_{\cj_K}(V,W) \times \mor_{\cj_K}(W,X) \to \mor_{\cj_K}(V,X)$$
$$((f \colon V \to W, U_{V,W}),(g \colon W \to X,U_{W,X}) \mapsto (g \circ f \colon V \to X, U_{W,X} + g(U_{V,W}))$$
Then by definition, independent of the choice of $K$, there are functors such as
\begin{itemize}
\item $V \mapsto \text{GL}(V)$
\item $V \mapsto \text{SL}(V)$
\end{itemize}
For $F = \mathbb{R}$ there is a functor
$$\cj \to \cj_{\mathbb{R}}$$
$$V \mapsto V$$
$$f \mapsto (f,(\im f)^\perp)$$
inducing (by precomposition) a calculus for functors on $\cj_\mathbb{R}$. \\
The question is if there are such calculi for other fields too. 
Speculatively, such calculi might be applied to questions of arithmetic, e.g.\ about $\text{SL}_n(F)$ or maybe even $\text{SL}_n(\mathbb{Z})$. 

We do not know the answer. We speculate that for local fields such as $\mathbb{Q}_p$ (on which there exists $p$-adic analysis), such a calculus might look rather similar to Orthogonal Calculus, while for even more general fields such as global or finite fields, we assume the proofs would need to use quite different, rather more algebro-geometric methods\footnote{
Could motivic homotopy theory play a role here?}.
\end{question}

\begin{question}
It seems very likely that \cref{orthogonal_monoidality_theorem} holds equally, with the same proof strategy, for Goodwillie Calculus for functors $\cC \to \cD$ when $\cC$ and $\cD$ are sufficiently well behaved, e.g.\ $\cD$ closed monoidal. Does this have interesting consequences in Goodwillie Calculus? What examples of monoidal functors to which Goodwillie Calculus can be applied are there?
\end{question}

{\sloppy

\cleardoublepage
\phantomsection
\addcontentsline{toc}{chapter}{Bibliography}
\printbibliography
}

\end{document}